\documentclass[12pt]{amsart}

\usepackage[all]{xypic}
\usepackage{tikz}
\usepackage{ragged2e}
\usetikzlibrary{arrows} 
\usetikzlibrary{decorations.markings}
\usepackage{graphicx}
\usepackage{bm}
\usepackage{float}
\usepackage{epsf}
\usepackage{verbatim} 
\usepackage{amsmath}
\usepackage{amsfonts}
\usepackage{amssymb}
\usepackage{mathrsfs}
\usepackage{amsthm}
\usepackage{newlfont}
\usepackage{enumitem}
\usepackage[new]{old-arrows}
\usepackage{booktabs}
\usepackage{enumitem}
\usepackage{makecell}
\usepackage[multiple]{footmisc}
\usepackage[answerdelayed,lastexercise]{exercise}
\usepackage[hidelinks]{hyperref}
\usepackage{mleftright}

\usepackage{fnpct}
\usepackage{colonequals} % From Poonen's "writing"
\usepackage{stmaryrd}

\usepackage{apptools}
\usepackage{chngcntr}
\newtheorem{thm}{Theorem}[section]
\newtheorem{prop}[thm]{Proposition}
\newtheorem{lem}[thm]{Lemma}
\newtheorem{cor}[thm]{Corollary}
\newtheorem{conj}[thm]{Conjecture}

\theoremstyle{definition}
\newtheorem{defn}[thm]{Definition}

\newtheorem{example}[thm]{Example}

\theoremstyle{remark}
\newtheorem{rem}[thm]{Remark}

\numberwithin{equation}{section}
\numberwithin{thm}{section}

\makeatletter 
\newcommand\mynobreakpar{\vspace{0.02in}\par\nobreak\@afterheading}  
\makeatother

\makeatletter
\@addtoreset{footnote}{section}
\makeatother

\usepackage{xcolor}
\hypersetup{
	colorlinks,
	linkcolor={red!50!black},
	citecolor={blue!50!black},
	urlcolor={blue!80!black}
}

\AtBeginDocument{%
	\def\MR#1{}
}

\DeclareMathOperator{\Hom}{Hom}
\DeclareMathOperator{\End}{End}
\DeclareMathOperator{\Aut}{Aut}

\DeclareMathOperator{\Ht}{ht}

\DeclareMathOperator{\Spec}{Spec}

\DeclareMathOperator{\chr}{char}

\DeclareMathOperator{\Mat}{Mat}
\DeclareMathOperator{\ord}{ord}

\DeclareMathOperator{\Vol}{Vol}
\DeclareMathOperator{\Gal}{Gal}
\DeclareMathOperator{\GL}{GL}
\DeclareMathOperator{\PGL}{PGL}

\DeclareMathOperator{\Tr}{Tr}

\DeclareMathOperator{\diag}{diag}
 
\DeclareMathOperator{\Nr}{Nr}

\DeclareMathOperator{\disc}{disc}

\newcommand{\fN}{\mathfrak{N}}
\newcommand{\fO}{\mathfrak{O}}
\newcommand{\fP}{\mathfrak{P}}

\newcommand{\fm}{\mathfrak{m}}
\newcommand{\fn}{\mathfrak{n}}

\newcommand{\fp}{\mathfrak{p}}
\newcommand{\fq}{\mathfrak{q}}

\newcommand{\cE}{\mathcal{E}}

\newcommand{\cM}{\mathcal{M}}

\newcommand{\cO}{\mathcal{O}}

\newcommand{\cT}{\mathcal{T}}

\newcommand{\A}{\mathbb{A}}

\newcommand{\C}{\mathbb{C}}

\newcommand{\F}{\mathbb{F}}

\newcommand{\PP}{\mathbb{P}}
\newcommand{\Q}{\mathbb{Q}}
\newcommand{\R}{\mathbb{R}}

\newcommand{\Z}{\mathbb{Z}}

\newcommand{\G}{\Gamma}

\newcommand{\To}{\longrightarrow}

\newcommand{\Fi}{F_\infty}

\newcommand{\oF}{\overline{\F}}

\newcommand{\La}{\Lambda}
\newcommand{\la}{\lambda}

\newcommand{\norm}[1]{\left\Vert#1\right\Vert}

\newcommand{\twist}[1]{#1\!\left\{\tau\right\}}

\newcommand{\atwist}[2]{#1\langle#2\rangle}

\newcommand{\abs}[1]{\left|#1\right|}

\title[Stabilization of isogeny spaces]{Stabilization of isogeny spaces between supersingular Drinfeld modules}

\author{Giacomo Micheli} 
\address{Department of Mathematics \& Statistics, The University of South Florida, Tampa, Florida, United States of America}
\email{gmicheli@usf.edu}

\author{Mihran Papikian}
\address{Department of Mathematics, Pennsylvania State University, University Park, Pennsylvania, United States of America}
\email{papikian@psu.edu}

\thanks{The first author was supported by NSF CAREER grant 2338424. The second author was supported in part by the Simons Foundation, award number MPS-TSM-00008093.} 


\begin{document}

\begin{abstract}
	Let $\fp$ be a prime of degree $d$ in $A = \F_q[T]$ and let $\phi, \psi$ be supersingular Drinfeld modules of rank $r \geq 2$ in $A$-characteristic $\fp$. We study the $\F_q$-dimension of the space
	\[
	\cM_s(\phi, \psi) = \{u \in \Hom(\phi, \psi) : \deg_\tau u \leq s\}
	\]
	as a function of $s$. By analyzing $\Hom(\phi, \psi)$ as a normed $A$-lattice in the local division algebra at $\infty$ via its successive minima, we obtain an exact closed-form expression for $\dim_{\F_q}\cM_s(\phi, \psi)$ valid for every $s \geq 0$, together with structural constraints on the successive-minima multiset which imply the stabilization formula
	\[
	\dim_{\F_q}\cM_s(\phi, \psi) = r(s+1) - \tfrac{r(r-1)(d-1)}{2}
	\]
	for all $s \geq r^2(r-1)(d-1)/2$. We conjecture that the optimal threshold is $s \geq (r-1)(d-1) - 1$, and prove this sharp form for $r = 2$ by independent automorphic methods, using the decomposition of a Brandt-type theta series on the Bruhat--Tits tree of $\PGL_2(\Fi)$ into Eisenstein and cuspidal parts together with the polynomiality of the cuspidal $L$-function. We also recast our results in Mornev's geometric framework, in which the conjecture becomes a cohomology-vanishing statement for a family of vector bundles on $\PP^1$, and illustrate the theory with explicit examples in which all successive-minima multisets permitted by our constraints are realized.
\end{abstract}
	
	\maketitle
	
	%=============================================
	
	\section{Introduction}
	
	\subsection*{Motivation and main results}
	Let $A = \F_q[T]$ and $F = \F_q(T)$. Fix a prime $\fp\lhd A$ of degree $d$ and let $\phi, \psi$ be supersingular Drinfeld modules of rank $r\geq 2$ in $A$-characteristic $\fp$ over $\oF_\fp$; definitions are recalled in Section~\ref{sSupersingular}. The space of isogenies $\Hom(\phi, \psi)$ is a free $A$-module of rank $r^2$, naturally filtered by the $\tau$-degree:
	\[
	\cM_s(\phi, \psi) \colonequals \{u \in \Hom(\phi, \psi) : \deg_\tau u \leq s\}, \qquad s\geq 0.
	\]
	Each $\cM_s$ is a finite-dimensional $\F_q$-vector space, and the basic question of this paper is the precise value of $\dim_{\F_q}\cM_s(\phi, \psi)$, and the extent to which it depends on the pair $(\phi, \psi)$ 
	rather than on the invariants $r$, $d$, $s$ alone.
	
	The question arose for us from rank-metric coding. For a prime $\fq\neq \fp$, every element of $\Hom(\phi, \psi)$ acts $A$-linearly on the $\fq$-torsion module $\phi[\fq]\cong \F_\fq^r$, and when $\deg(\fq) > s$ this gives an embedding
	\[
	\cM_s(\phi, \psi) \hookrightarrow \Hom_{\F_\fq}(\phi[\fq], \psi[\fq])\cong \Mat_r(\F_\fq)
	\]
	in which every nonzero element acts invertibly, i.e., a rank-metric code of maximum rank distance. Such a code attains the Singleton bound for rank-metric codes, so is a \emph{semifield code} in the sense of \cite{MP1}, precisely when $\dim_{\F_q}\cM_s = r\deg(\fq)$. The supersingular hypothesis makes $\Hom(\phi, \psi)$ as large as possible (free of rank $r^2$ over $A$), but for $d\geq 2$ the dimension falls short of this benchmark by a positive amount. Pinning down that shortfall, and how it varies with $\phi, \psi$, is the central numerical question.
	
	Our first main result gives an explicit answer in all ranks.
	
	\begin{thm}\label{thmA}
		For every $s\geq r^2(r-1)(d-1)/2$,
		\[
		\dim_{\F_q}\cM_s(\phi, \psi) = r(s+1) - \frac{r(r-1)(d-1)}{2}.
		\]
	\end{thm}
	
	The threshold $r^2(r-1)(d-1)/2$ is not optimal: we conjecture (Conjecture~\ref{conj:emax-bound}) that the same formula holds in the sharper range $s\geq (r-1)(d-1) - 1$. We establish this sharp range for rank $2$ by independent automorphic methods.
	
	\begin{thm}\label{thmB}
		Let $r = 2$. For every $s\geq d - 2$,
		\[
		\dim_{\F_q}\cM_s(\phi, \psi) = 2(s+1) - (d-1).
		\]
	\end{thm}
	
	\subsection*{Successive minima of \texorpdfstring{$\Hom(\phi, \psi)$}{Hom(phi,psi)}}
	
	The proof of Theorem~\ref{thmA} passes through a detailed analysis of $\Hom(\phi, \psi)$ as a normed $A$-lattice, and the structural results that emerge are, we believe, of independent interest. The norm $\norm{u}_D \colonequals q^{\deg_\tau(u)/r}$ extends to the unique ultrametric absolute value on the central division algebra $D_\infty$ of dimension $r^2$ and invariant $-1/r$ over the completion $\Fi$ of $F$ at $1/T$, and by a theorem of Taguchi~\cite[\S 4]{Taguchi}, $\Hom(\phi, \psi)$ admits an $A$-basis $u_1, \ldots, u_{r^2}$ realizing the successive minima of this lattice. Writing $e_k = \deg_\tau(u_k)$ with $e_1 \leq \cdots \leq e_{r^2}$, the multiset $\{e_k\}$ is an invariant of $\Hom(\phi, \psi)$ independent of the basis, and an SMB is characterized by the property of \emph{norm-orthogonality}: the ultrametric inequality $\norm{c_1 u_1 + \cdots + c_{r^2} u_{r^2}}_D = \max_k \abs{c_k}\norm{u_k}_D$ holds for all $c_k\in \Fi$, without cancellation at the maximum.
	
	Norm-orthogonality converts the count of $\cM_s$ into a sum of independent one-dimensional counts and yields a closed-form dimension formula valid for every $s\geq 0$:
	\begin{equation}\label{eq:intro-dimMs}
		\dim_{\F_q}\cM_s(\phi, \psi) = \sum_{k=1}^{r^2} \max\!\left(0,\ \left\lfloor\frac{s - e_k}{r}\right\rfloor + 1\right). 
	\end{equation}
	The stable formula of Theorem~\ref{thmA} is the case $s\geq e_{\max}(\La) - r + 1$ of \eqref{eq:intro-dimMs}.
	
	We also obtain three structural constraints on the multiset $\{e_k\}$, which together determine its behavior almost completely:
	\begin{enumerate}
		\item[(1)] $\sum_{k=1}^{r^2} e_k = r^2(r-1)d/2$ (Proposition~\ref{prop:sum-ek}), proved from the discriminant identity for the maximal $A$-order $\End(\phi)$ via the trace pairing;
		\item[(2)] $\{e_k \bmod r\}$ contains each residue class $0, 1, \ldots, r-1$ exactly $r$ times (Proposition~\ref{prop:uniform-residues}), reflecting the unramified filtration of the local maximal order in $D_\infty$;
		\item[(3)] when $\La = \End(\phi)$ with $\phi_T = t + g_1\tau + \cdots + g_r\tau^r$, the number of zero minima equals $\gcd\{i : g_i \neq 0\}$ (Lemma~\ref{lem:e1-zero}); in particular $e_1 = 0$, and for $r$ prime the first $r$ minima all vanish precisely when $\phi_T = t + g\tau^r$.
	\end{enumerate}
	Conjecture~\ref{conj:emax-bound} adds a fourth constraint, $e_{\max}(\La) \leq (r-1)d$. We prove it unconditionally for $r = 2$ (Corollary~\ref{cor:emax-r2}) by combining \eqref{eq:intro-dimMs} with Theorem~\ref{thmB}, and we verify it with equality for the family $\phi_T = t + \tau^r$ with $\gcd(r,d) = 1$ (Proposition~\ref{prop:phi-vanishing}). Constraints (1)--(3) alone (without the conjecture) already force the unconditional crude bound $e_{\max}(\La) \leq r^2(r-1)(d-1)/2 + (r-1)$ (Proposition~\ref{prop:emax-crude}), and this is the bound that yields Theorem~\ref{thmA}.
	
	These constraints are tight. In Section~\ref{sExample} we tabulate, for the supersingular rank-$2$ classes at two small primes ($q = 3$ with $d = 4$ and $d = 5$), the multisets $\{e_k\}$ that arise as $\phi$ varies among the isomorphism classes; in both cases, every multiset consistent with (1)--(3) and the bound $e_{\max}\leq (r-1)d$ is realized, and each multiset corresponds to a structurally distinguished class of supersingular Drinfeld modules characterized by an arithmetic condition on the parameter.
	
	\medskip
	
	Theorem~\ref{thmB} is proved by an entirely different route. We encode the dimensions in a Brandt-type theta series on the Bruhat--Tits tree of $\PGL_2(\Fi)$ whose decomposition into Eisenstein and cuspidal parts produces a recursion that determines $\dim_{\F_q}\cM_s$ exactly; the crucial input is the polynomiality of the cuspidal $L$-function, a consequence of deep results of Drinfeld, Deligne, and Grothendieck. A third, geometric perspective due to Mornev~\cite{Mornev} expresses $\dim_{\F_q}\cM_s$ in terms of the volume of $\Hom(\phi, \psi)$ as an $A$-lattice and the cohomology of an associated vector bundle on $\PP^1_{\F_q}$. Under this dictionary, Conjecture~\ref{conj:emax-bound} becomes a cohomology-vanishing statement (Corollary~\ref{cor:cohom-vanishing}), and the threshold of Theorem~\ref{thmB} is the range in which the relevant $H^1$ vanishes.
	
	Theorems~\ref{thmA} and \ref{thmB} have direct applications to the construction of rank-metric codes from supersingular Drinfeld modules. The shortfall of $\dim_{\F_q}\cM_s$ from the semifield value $r(s+1)$ is the constant $r(r-1)(d-1)/2$, independent of $s$, so the codes obtained from the embedding $\cM_s\hookrightarrow \Mat_r(\F_\fq)$ at $\deg(\fq) = s+1$ are asymptotically MRD: the ratio of the code dimension to the Singleton-bound dimension tends to $1$ as $s\to\infty$. More elaborate constructions exploit the precise dimension formula to produce rank-metric codes whose rate \emph{exceeds} the MRD ceiling, at the cost of an asymptotically vanishing fraction of rank-deficient codewords. These constructions will be developed in forthcoming joint work.
	
	\subsection*{Organization}
	
	Section~\ref{sSupersingular} sets up notation and collects the facts on supersingular Drinfeld modules used throughout. Section~\ref{sec:smb-stabilization} develops the successive-minima analysis and proves Theorem~\ref{thmA}. Section~\ref{sec:mornev-volume} gives the geometric reinterpretation via Mornev's volume and recasts the conjecture as cohomology vanishing on $\PP^1$. Section~\ref{secBrandt} carries out the automorphic argument and proves Theorem~\ref{thmB}. Section~\ref{sExample} presents explicit examples: a closed-form successive-minima basis for the case $\phi_T = t + \tau^r$ (recovering Theorem~\ref{thmA} for this family by elementary means via the Sylvester--Frobenius theorem), and a tabulation of pre-stable dimension tuples for the rank-$2$ supersingular classes at two small primes, in which the SMB multisets permitted by constraints (1)--(3) and the bound $e_{\max}\leq d$ are all realized.
	
	%==============================================
	
	\section{Supersingular Drinfeld modules}\label{sSupersingular}
	
	In this section we set up the notation used throughout the paper and collect the facts about supersingular Drinfeld modules that we will need. 
	
	\subsection{Notation and basic definitions}\label{ssNotation}
	We denote by $A=\F_q[T]$ the polynomial ring in indeterminate $T$ with coefficients in $\F_q$.  Let $A_+$
	denote the set of monic nonzero polynomials in $A$.
	Let $F=\F_q(T)$ be the fraction field of $A=\F_q[T]$.
	Given a nonzero ideal $\fn\lhd A$, we denote by the same symbol
	the unique monic generator of $\fn$. The \textit{primes} of $A$ are the maximal ideals of $A$.
	Given a prime $\fp\lhd A$, we write $\F_\fp=A/\fp$.
	
	Let $k$ be a finite field equipped with an $\F_q$-algebra homomorphism $\gamma\colon A\to k$. The \textit{$A$-characteristic of $k$} is
	$\chr_A(k)\colonequals \ker(\gamma)$; it is a prime of $A$. We denote $\fp=\chr_A(k)$ and $t=\gamma(T)$.
	
	Let $\twist{k}$ be the \textit{twisted polynomial ring}, i.e., the noncommutative
	polynomial ring in $\tau$ with coefficients in $k$, subject to the commutation rule $\tau a=a^q\tau$ for all $a\in k$. 
	For $u=a_h\tau^h+\cdots +a_n\tau^n\in \twist{k}$ with $0\leq h\leq n$, $a_h\neq 0$ and $a_n\neq 0$, the \textit{height} of $u$ is $\Ht(u)=h$ and the \textit{degree} of $u$ is $\deg_\tau(u)=n$.  
	The map $f=\sum_{i=0}^n a_i \tau^i \longmapsto f(x)=\sum_{i=0}^n a_i x^{q^i}$ gives a ring isomorphism between $\twist{k}$
	and the ring $\atwist{k}{x}$ of $\F_q$-linear polynomials, where multiplication on $\atwist{k}{x}$ is defined by composition.
	
	A \textit{Drinfeld module} of rank $r\geq 1$ over $k$ is an $\F_q$-algebra homomorphism
	\begin{align*}
		\phi\colon A &\To \twist{k}, \\
		a &\longmapsto \phi_a=\gamma(a)+g_1(a)\tau+\cdots +g_n(a)\tau^n,
	\end{align*}
	such that for $a\neq 0$ we have $n=\deg(a) r$ and $g_n(a)\neq 0$. Note that
	$\phi$ is uniquely determined by $\phi_T$, so to define a Drinfeld module over $k$
	one simply chooses $g_1, \dots, g_r\in k$ with $g_r\neq 0$ and sets $\phi_T=t+g_1\tau+\cdots+g_r\tau^r$.
	The \textit{height} of $\phi$ is defined as $H(\phi)=\Ht(\phi_\fp)/\deg(\fp)$; this is a positive
	integer satisfying $1\leq H(\phi)\leq r$ (see \cite[Lem.\ 3.2.11]{PapikianGTM}).
	
	Through $\phi$, $\bar{k}$ acquires an $A$-module structure, where $a\in A$ acts on $\beta\in \bar{k}$ by $a\ast \beta=\phi_a(\beta)$. Denote this module by ${^\phi}{\bar{k}}$. Given $0\neq a\in A$,
	the \textit{$a$-torsion submodule} $\phi[a]$ is the set of roots of $\phi_a(x)$; it is
	an $A$-submodule of ${^\phi}{\bar{k}}$.
	If $\fp$ does not divide $a$, then $\phi[a]\cong (A/aA)^r$; see \cite[Cor.\ 3.5.3]{PapikianGTM}.
	
	The group of \textit{morphisms $\phi\to \psi$} between two Drinfeld modules defined over $k$ is
	\begin{align*}
		\Hom_k(\phi, \psi) &=\{u\in \twist{k}\mid u\phi_a=\psi_au\text{ for all }a\in A\} \\
		& = \{u\in \twist{k}\mid u\phi_T=\psi_Tu\}.
	\end{align*}
	The \textit{endomorphism ring} of $\phi$ is
	$
	\End_k(\phi) = \Hom_k(\phi, \phi).
	$
	For simplicity, we denote $\Hom(\phi, \psi)=\Hom_{\bar{k}}(\phi, \psi)$ -- the group of all possible morphisms $\phi\to \psi$
	over the algebraic closure $\bar{k}$ of $k$.
	Given $u\in \Hom(\phi, \psi)$ and $a\in A$, we define $a\circ u=u\phi_a=\psi_au$.
	It is easy to check that $a\circ u \in \Hom(\phi, \psi)$, so $\Hom(\phi, \psi)$ is an $A$-module. It is known
	that $\Hom(\phi, \psi)$ is a free $A$-module of rank $\leq r^2$; see \cite[Thm. 3.4.1]{PapikianGTM}.
	A nonzero morphism $u\in \Hom(\phi, \psi)$ is called an \textit{isogeny}.
	An isogeny $u$ is an \textit{isomorphism} if it is invertible in $\twist{k}$, i.e., $u\in k^\times$.

	Given two Drinfeld modules $\phi$ and $\psi$ over $k$ and an isogeny $u\in \Hom(\phi, \psi)$, let
	\[
	\ker(u)\colonequals \{\alpha\in \bar{k}\mid u(\alpha)=0\},
	\]
	which is the set of distinct roots of the polynomial $u(x)\in \atwist{k}{x}$. This is a finite $A$-module, with the action of $A$ defined by $a\circ \beta  =\phi_a(\beta)$.
	For a finite $A$-module $M$ we define $\chi(M)$ as the product of the invariant factors of $M$, which are assumed to be monic.
	The \textit{norm} of $u$ is (cf. \cite[p. 194]{Gekeler91})
	\begin{equation}\label{eqNormu}
		\fN(u)=
		\fp^{\Ht(u)/\deg(\fp)}\cdot \chi(\ker u).
	\end{equation}
	By \cite[Prop. 3.3.4]{PapikianGTM},  $\deg(\fp)$ divides $\Ht(u)$. Thus, $\fN(u)$ is a monic
	polynomial in $A$. It is immediate from the definitions that
	\begin{equation}\label{eqtaudeg}
		\deg_\tau(u)=\deg_T\fN(u).
	\end{equation}

	\subsection{Supersingular Drinfeld modules}\label{ssSSDM}
	We now recall the definition of supersingularity and its consequences for the endomorphism ring.
	
	\begin{defn}
		A Drinfeld module $\phi$ over $k$ of rank $r$ is called
		\textit{supersingular} if it satisfies one of the following equivalent conditions
		(see \cite[Sec. 4]{Gekeler91}):
		\begin{enumerate}
			\item $H(\phi)=r$.
			\item $\phi[\fp]=0$.
			\item $\dim_F(\End(\phi)\otimes_A F)=r^2$.
		\end{enumerate}
	\end{defn}
	
	\begin{lem}\label{lemEndomSS}
		A supersingular Drinfeld module $\psi$ of rank $r$ over $\bar{k}$ is isomorphic to a supersingular
		Drinfeld module $\phi$ defined over the degree $r$ extension $\F_{\fp^r}$ of $\F_\fp$ and such that all endomorphisms
		of $\phi$ over $\bar{k}$ are already defined over $\F_{\fp^r}$.
	\end{lem}
	\begin{proof} The following argument is very similar to the argument of the proof of \cite[Prop. 4.2]{Gekeler91}.
		
		Suppose $\psi_T=t+g_1\tau+\cdots+g_r\tau^r$. Let $\alpha$ be a fixed $(q^r-1)$-th root of $g_r$ and
		let $\phi=\alpha \psi\alpha^{-1}$. Then $\phi_T=t+f_1\tau+\cdots+\tau^r$, and $\phi_\fp=\tau^{rd}$. Let $u=u_0+u_1\tau+\cdots+u_s\tau^s$
		be an endomorphism of $\phi$ over $\bar{k}$. The relation $\phi_\fp u=u \phi_\fp$
		implies $u_i^{q^{rd}}=u_i$ for all $0\leq i\leq s$. Thus, $u\in \twist{\F_{\fp^r}}$. Since every $\phi_a$ lies in $\End(\phi)$, we conclude that
		$\phi$ itself is defined over $\F_{\fp^r}$.
	\end{proof}
	
	The previous lemma implies that the number of supersingular Drinfeld modules of rank $r$ over $\bar{k}$, up to isomorphism, is finite.
	Moreover, it is known that any two supersingular Drinfeld modules of rank $r$ over $\bar{k}$ are isogenous; see \cite[Lem. 4.4.3]{PapikianGTM}.
	
	The endomorphism ring of a supersingular Drinfeld module enjoys the following
	maximality property, due to Gekeler \cite[Thm.\ 4.3]{Gekeler91}:
	
	\begin{thm}\label{thmGekelerMaxOrder}
		Let $\phi$ be a supersingular rank $r$ Drinfeld module over $\bar{k}$.
		\begin{enumerate}
			\item $D=\End(\phi)\otimes_A F$ is a central division algebra over $F$ of dimension $r^2$
			with invariants $1/r$ and $-1/r$ at $\fp$ and $\infty$, respectively, and $0$ at all other places.
			\item $\End(\phi)$ is a maximal $A$-order in $D$.
			\item The left ideal classes of $\End(\phi)$
			are in bijection with the isomorphism classes of supersingular rank $r$
			Drinfeld modules in characteristic $\fp$.
		\end{enumerate}
	\end{thm}
	
	\subsection{Supersingular Drinfeld modules of rank \texorpdfstring{$2$}{2}}\label{ssSSrank2}
	In this subsection we record several results that are special to the rank-$2$ case. These are not used in the general development of Sections \ref{sec:smb-stabilization}--\ref{sec:mornev-volume}, but they will be needed for the explicit computations carried out in Section \ref{sExample}.
	
	We first recall that given a Drinfeld module of rank $2$ defined by
	\[
	\phi_T=t+g\tau+\Delta\tau^2,
	\]
	its \textit{$j$-invariant} is $j(\phi)=g^{q+1}/\Delta$. It is easy to check that two Drinfeld modules $\phi, \psi$ of rank $2$ are isomorphic over $\bar{k}$ if and only if $j(\phi)=j(\psi)$.
	
	\begin{lem}\label{lemAut}
		Let $\phi$ be a Drinfeld module of rank $2$ over $\bar{k}$. Then
		\[
		\Aut(\phi) =
		\begin{cases}
			\F_q^\times & \text{if } j(\phi) \neq 0, \\
			\F_{q^2}^\times & \text{if } j(\phi) = 0.
		\end{cases}
		\]
	\end{lem}
	
	\begin{proof} %See \cite[Cor. 3.8.6]{PapikianGTM}.
		Write $\phi_T = t + g\tau + \Delta\tau^2$. An element $c \in \bar{k}^\times$ is an automorphism of $\phi$ if and only if $c\phi_T = \phi_T c$, i.e., $cg = gc^q$ and $c\Delta = \Delta c^{q^2}$. The second condition forces $c \in \F_{q^2}^\times$; when $g \neq 0$ (equivalently $j(\phi) \neq 0$), the first refines this to $c \in \F_q^\times$.
	\end{proof}
	
	\begin{defn}\label{defHp}
		Consider the Drinfeld module
		\begin{equation}\label{eqLegendre}
			\phi_T=t+\tau+x \tau^2
		\end{equation}
		over $\F_\fp[x]$, where $x\neq 0$ is an indeterminate. (The structure map $\gamma\colon A\to \F_\fp[x]$ is the composition $A\to \F_\fp\hookrightarrow \F_\fp[x]$.) Let $H_\fp(x)\in \F_\fp[x]$ be the coefficient of $\tau^d$ in $\phi_\fp$ computed in $\twist{\F_\fp[x]}$.
		These polynomials are Gekeler's Deuring polynomials \cite{Gekeler83}, differently normalized.
	\end{defn}
	
	\begin{thm}\hfill
		\begin{enumerate}
			\item	The polynomial $H_\fp(x)$ has degree
			\[
			\begin{cases}
				\frac{q^d-1}{q^2-1}, & \text{if $d$ is even}, \\
				\frac{q^d-q}{q^2-1} , & \text{if $d$ is odd}.
			\end{cases}
			\]
			\item The polynomial $H_\fp(x)$ splits over $\F_\fp$ into a product of distinct linear factors and irreducible
			quadratics.
		\end{enumerate}
	\end{thm}
	\begin{proof}
		The roots of $H_\fp(x)$ are the reciprocals of non-zero supersingular $j$-invariants in $\oF_\fp$.
		Since all such $j$'s are in $\F_{\fp^2}$, the claim follows from the fact that $j=0$ is supersingular
		if and only if $d$ is odd (see Lemma \ref{lemDSpecial}) and the results in \cite{Gekeler83} (reproduced in English in \cite[pp. 266-267]{PapikianGTM}).
	\end{proof}
	
	\begin{example}\label{Example1}
		Let $q=3$ and $\frak p(T)=T^4+T+2$. Then
		\[
		\begin{aligned}
			H_\fp(x)=\;&t^{64}x^{10}+t^{6}x^{9}+t^{15}x^{3}+t^{58}x+1\\
			=\;&t^{64}(x+t^{50})(x+t^{62})(x+t^{66})(x+t^{78})\\
			&\cdot (x^2+t^{3}x+t^{8})(x^2+t^{12}x+t^{12})(x^2+t^{33}x+t^{60}).
		\end{aligned}
		\]
	\end{example}
	
	\begin{prop}\label{lemLegendre}
		Let $x_0, y_0\in \F_{\fp^2}$ be two roots of $H_\fp(x)$, and let
		$\phi_T=t+\tau+x_0\tau^2$ and $\psi_T=t+\tau+y_0\tau^2$
		be the corresponding supersingular Drinfeld modules defined over $\F_{\fp^2}$.
		Denote $\alpha = \Nr_{\F_{\fp^2}/\F_{q^2}}(x_0)$ and $\beta=\Nr_{\F_{\fp^2}/\F_{q^2}}(y_0)$.
		\begin{enumerate}
			\item The elements $\alpha$ and $\beta$ lie in $\F_q$.
			\item All morphisms $\phi\to \psi$ are defined over $\F_{\fp^2}$ if and only if $\alpha=\beta$. In particular, all
			endomorphisms of $\phi$ are defined over $\F_{\fp^2}$.
			\item In general, all morphisms $\phi\to \psi$ are defined over the $(q-1)$-th degree extension of $\F_{\fp^2}$, i.e., over $\F_{q^{2d(q-1)}}$.
		\end{enumerate}
	\end{prop}
	\begin{proof}
		Comparing
		the coefficient of $\tau^{2d+1}$ on both sides of $\phi_T\phi_\fp=\phi_\fp\phi_T$, we get $\alpha^q=\alpha$.
		Thus, $\alpha\in \F_q$. Similarly, $\beta\in \F_q$. This proves (1).
		
		Let $u=u_0+u_1\tau+\cdots+u_s\tau^s$ be a morphism $\phi\to \psi$.
		From $u \phi_\fp = \psi_\fp u$ we get $\alpha u_i = \beta u_i^{q^{2d}}$ for all $0\leq i\leq s$, because $\alpha, \beta\in \F_q$.
		Thus,  $u_i = u_i^{q^{2d}}$ for all $0\leq i\leq s$ if and only if $\alpha=\beta$. This proves (2).
		
		In general, when $\alpha$ is not necessarily equal to $\beta$, for $u_i\neq 0$ we get $u_i^{(q^{2d}-1)(q-1)}=1$.
		We claim that if $b\in \oF_q$ satisfies $b^{(q^{2d}-1)(q-1)}=1$,
		then $b \in \F_{q^{2d(q-1)}}$, and moreover $\F_{q^{2d(q-1)}}$ is the smallest extension of
		$\F_q$ containing all such~$b$.
		
		Set $N = (q^{2d}-1)(q-1)$.  We need to find the smallest $n$ such that
		$N \mid q^n - 1$. A necessary condition is
		$(q^{2d}-1) \mid (q^n-1)$, which holds if and only if $2d \mid n$.
		Write $n = 2dm$ for some positive integer~$m$.  Then $N \mid (q^{2dm}-1)$ if and only if $(q-1)$ divides
		\[
		\frac{q^{2dm}-1}{q^{2d}-1} = 1 + q^{2d} + q^{4d} + \cdots + q^{2d(m-1)}.
		\]
		Since each term $q^{2dj}$ is congruent to $1$ modulo $q-1$, we have
		\[
		\frac{q^{2dm}-1}{q^{2d}-1} \equiv m \pmod{q-1}.
		\]
		The smallest positive $m$ divisible by $q-1$ is $q-1$, giving $n = 2d(q-1)$.
	\end{proof}
	
	\begin{rem}
		In Example \ref{Example1}, all
		Drinfeld modules defined over $\F_{\fp}$ are isogenous over $\F_{\fp^2}$, but those Drinfeld modules which are defined over $\F_{\fp^2}$
		properly are not necessarily isogenous over $\F_{\fp^2}$.
	\end{rem}
	
	%===============================================================
	
	%==================================================================
	\section{Stabilization via successive minima}\label{sec:smb-stabilization}
	
	In this section we prove that the dimension $\dim_{\F_q}\cM_s(\phi,\psi)$ stabilizes to an 
	explicit linear function of $s$ for all $s$ beyond an effective, provable threshold. For supersingular Drinfeld modules $\phi, \psi$ of rank $r \geq 2$ in $A$-characteristic $\fp$ of degree $d$, we obtain a closed formula for $\dim_{\F_q}\cM_s$ valid for every $s \geq 0$, and we show that
	\begin{equation}\label{eq-conj}
		\dim_{\F_q} \cM_s(\phi,\psi) = r(s+1) - \tfrac{r(r-1)(d-1)}{2} \qquad \text{for all } s \geq \tfrac{r^2(r-1)(d-1)}{2}.
	\end{equation}
	This range is unconditional and effective in all ranks (Proposition~\ref{prop:emax-crude}). We conjecture that it can be sharpened to the optimal range $s \geq (r-1)(d-1) - 1$ (Conjecture~\ref{conj:emax-bound}); this sharper range is established for $r = 2$ by the independent automorphic methods of Section~\ref{secBrandt}.
	
	The argument is structural. We realize $\La \colonequals \End(\phi)$ (or more generally $\Hom(\phi,\psi)$) as a normed $A$-lattice, the norm coming from the reduced-norm valuation of the local division algebra $D_\infty$. A successive-minima basis (SMB) of $\La$ is automatically norm-orthogonal, which converts the count of $\cM_s$ into a finite-dimensional polynomial counting problem and yields a closed formula for $\dim_{\F_q}\cM_s$ in terms of the SMB $\tau$-degrees alone. Combining this with two structural facts about SMBs of maximal orders---an equidistribution-modulo-$r$ statement and the evaluation of $\sum_k \deg_\tau(u_k)$ from the discriminant of $\La$---we obtain \eqref{eq-conj}. The geometric meaning of these invariants is taken up in Section~\ref{sec:mornev-volume}.
	
	Throughout this section, $\Fi$ denotes the completion of $F$ at the place $\infty$ corresponding to $1/T$, $\cO_\infty$ its ring of integers, and $\varpi_\infty$ a fixed uniformizer of $\Fi$. The normalized valuation $\ord_\infty\colon \Fi^\times \to \Z$ satisfies $\ord_\infty|_{A\setminus\{0\}} = -\deg$. The corresponding absolute value is $\abs{\cdot} \colonequals q^{-\ord_\infty(\cdot)}$; throughout this section $\abs{\cdot}$ always denotes the absolute value at $\infty$, and the subscript will be omitted.
	
	%%%%%%%%%%%%%%%%%%%%%%%%%%%%%%%%%%%%%%%%%%%%%%%
	\subsection{The lattice structure on \texorpdfstring{$\End(\phi)$}{End(phi)}}\label{sub:lattice-structure}
	%%%%%%%%%%%%%%%%%%%%%%%%%%%%%%%%%%%%%%%%%%%%%%%
	
	Let $\phi$ be a supersingular Drinfeld module of rank $r$ over $\oF_\fp$ and let $\La \colonequals \End(\phi)$. By Theorem~\ref{thmGekelerMaxOrder}, $D = \La \otimes_A F$ is the central division algebra over $F$ of dimension $r^2$ with the prescribed local invariants, and $\La$ is a maximal $A$-order in $D$. Let
	\[
	V := \La \otimes_A \Fi = D \otimes_F \Fi = D_\infty,
	\]
	which is the unique central division algebra over $\Fi$ of dimension $r^2$ with invariant $-1/r$.
	
	The valuation $\ord_\infty$ extends uniquely to a $\Q$-valued valuation $w$ on $D_\infty$ via
	\[
	w(x) := \tfrac{1}{r} \ord_\infty(\Nr(x)), \qquad w(0) := +\infty,
	\]
	with value group $\tfrac{1}{r}\Z$ on $D_\infty^\times$, where $\Nr\colon D_\infty^\times\to \Fi^\times$ is the reduced norm. The maximal order in $D_\infty$ is
	\[
	\fO_\infty := \{x \in D_\infty : w(x) \geq 0\},
	\]
	its unique maximal two-sided ideal is $\fP_\infty := \{x : w(x) \geq 1/r\}$, and $\fP_\infty^r = \varpi_\infty \fO_\infty$. The residue ring $\fO_\infty/\fP_\infty$ is the field $\F_{q^r}$. 
	
	For $u\in \La$, by \cite[Lem.~3.10]{Gekeler91} we have
	\[
	\deg_\tau(u) = \deg_T (\Nr(u)) = - \ord_\infty(\Nr(u)) = -r \cdot w(u).
	\]
	
	\begin{defn}\label{def:norm}
		Define $\norm{\cdot}_D \colon V \to \R_{\geq 0}$ by
		\[
		\norm{x}_D := \abs{\Nr(x)}^{1/r} = q^{-w(x)}.
		\]
		On $\La$, this restricts to $\|u\|_D = q^{\deg_\tau(u)/r}$.
	\end{defn}
	
	The function $\norm{\cdot}_D$ is the unique ultrametric absolute value on $D_\infty$ extending $\abs{\cdot}$ on $\Fi$; it is multiplicative, $\norm{xy}_D=\norm{x}_D\norm{y}_D$.
	
	The pair $(\La, \norm{\cdot}_D)$ is an $A$-lattice in $V$ in the sense of \cite[Def.~1.3.2]{Mornev}:
	\begin{itemize}
		\item $\La$ is a free $A$-module of rank $r^2$, which generates $V$ over $\Fi$.
		\item The set $\{u \in \La : \norm{u}_D \leq R\}$ equals $\{u \in \La : \deg_\tau u \leq r \log_q R\}$, a finite-dimensional $\F_q$-subspace of $\La$. Therefore $\norm{\cdot}_D$ induces the discrete topology on $\La$.
	\end{itemize}
	
	The same definitions extend to $\Hom(\phi,\psi)$ for any pair of supersingular $\phi, \psi$ of rank $r$ in $A$-characteristic $\fp$. Indeed, $\Hom(\phi,\psi)$ is an invertible $(\End(\psi), \End(\phi))$-bimodule, free of rank $r^2$ as an $A$-module on either side. After $\otimes_A F$ it becomes free of rank $1$ as a left (resp.\ right) $D$-module: any nonzero isogeny $f \in \Hom(\phi,\psi)$ is a unit in $D$, and the map $u \mapsto f^{-1} u$ identifies $\Hom(\phi,\psi) \otimes_A F$ with $D$ as right $D$-modules. The norm $\norm{\cdot}_D = \abs{\Nr(\cdot)}^{1/r}$ extends to $\Hom(\phi,\psi) \otimes_A \Fi = D_\infty$.
	
	In what follows, by \emph{lattice} we mean either $\La = \End(\phi)$ or $\La = \Hom(\phi,\psi)$ for supersingular $\phi, \psi$ of rank $r$ over $\oF_\fp$.
	
	%%%%%%%%%%%%%%%%%%%%%%%%%%%%%%%%%%%%%%%%%%%%%%%
	\subsection{Successive-minima bases and the unit ball}\label{sub:smb-unit-ball}
	%%%%%%%%%%%%%%%%%%%%%%%%%%%%%%%%%%%%%%%%%%%%%%%
	
	A \emph{successive-minima basis} (SMB) of $\La$ is an $A$-basis $u_1, \ldots, u_{r^2}$ realizing the successive minima of $(\La, \norm{\cdot}_D)$, in the sense of \cite[\S 4]{Taguchi}. We write
	\[
	e_k \colonequals \deg_\tau(u_k), \qquad \norm{u_k}_D = q^{e_k/r},
	\]
	and order the basis so that $e_1 \leq e_2 \leq \cdots \leq e_{r^2}$, setting $e_{\max}(\La) := \max_k e_k = e_{r^2}$. The multiset $\{e_k\}_{k=1}^{r^2}$ is an invariant of $\La$, independent of the choice of SMB; see \cite[Lem.~4.2]{Taguchi}.
	
	The defining property of an SMB is \emph{norm-orthogonality}: by \cite[Lem.~4.2]{Taguchi}, for all $c_1, \ldots, c_{r^2} \in \Fi$,
	\begin{equation}\label{eq:norm-orthog}
		\norm{\sum_k c_k u_k}_D = \max_k \abs{c_k}\cdot \norm{u_k}_D.
	\end{equation}
	In particular, the ultrametric inequality is an equality on coordinate expansions in an SMB: no cancellation can occur at the maximum.
	
	\begin{lem}\label{lem:scaled-SMB-basis}
		Let $u_1, \ldots, u_{r^2}$ be an SMB of $\La$ with $\norm{u_k}_D = q^{e_k/r}$. For each $k$, choose $a_k \in \Fi^\times$ with $\ord_\infty(a_k) = \lceil e_k/r \rceil$. Then $\{a_k u_k\}_{k=1}^{r^2}$ is an $\cO_\infty$-basis of the maximal order $\fO_\infty$.
	\end{lem}
	
	\begin{proof}
		We have
		\[
		\norm{a_k u_k}_D = \abs{a_k}\cdot q^{e_k/r} = q^{e_k/r - \lceil e_k/r \rceil} \leq 1,
		\]
		so each $a_k u_k \in \fO_\infty$.
		
		Let $x \in \fO_\infty$ and write $x = \sum_k c_k u_k$ with $c_k \in \Fi$. By \eqref{eq:norm-orthog},
		\[
		1 \geq \norm{x}_D = \max_k \abs{c_k}\cdot q^{e_k/r},
		\]
		so $\abs{c_k} \leq q^{-e_k/r}$, i.e., $\ord_\infty(c_k) \geq e_k/r$. Since $\ord_\infty(c_k) \in \Z$, this forces $\ord_\infty(c_k) \geq \lceil e_k/r \rceil = \ord_\infty(a_k)$, hence $c_k/a_k \in \cO_\infty$. Therefore $x = \sum_k (c_k/a_k)(a_k u_k)$ is an $\cO_\infty$-linear combination of the $a_k u_k$.
		
		Finally, if $\sum_k b_k (a_k u_k) = 0$ with $b_k \in \cO_\infty$, then $\Fi$-linear independence of $\{u_k\}$ gives $b_k a_k = 0$, hence $b_k = 0$.
	\end{proof}
	
	By construction, $\ord_\infty(a_k) = \lceil e_k/r \rceil$ gives $w(a_k u_k) = \lceil e_k/r \rceil - e_k/r \in [0,1)\cap \tfrac{1}{r}\Z$, so
	\[
	a_k u_k \in \fP_\infty^{j_k} \setminus \fP_\infty^{j_k+1}, \qquad j_k \colonequals (-e_k) \bmod r \in \{0, 1, \ldots, r-1\}.
	\]
	
	\begin{prop}\label{prop:uniform-residues}
		For every lattice $\La$, the multiset $\{e_k \bmod r\}_{k=1}^{r^2}$ contains each residue class $0, 1, \ldots, r-1$ exactly $r$ times.
	\end{prop}
	
	\begin{proof}
		The $\cO_\infty$-module $\fO_\infty$ admits the filtration
		\[
		\fO_\infty \supset \fP_\infty \supset \fP_\infty^2 \supset \cdots \supset \fP_\infty^{r-1} \supset \fP_\infty^r = \varpi_\infty \fO_\infty.
		\]
		Each graded piece $\fP_\infty^j/\fP_\infty^{j+1}$ is one-dimensional over $\fO_\infty/\fP_\infty = \F_{q^r}$, hence has $\F_q$-dimension $r$. The reduction $\fO_\infty/\varpi_\infty \fO_\infty$ is therefore an $\F_q$-vector space of dimension $r^2$, carrying the induced filtration whose $r$ graded pieces each contribute dimension $r$.
		
		By Lemma~\ref{lem:scaled-SMB-basis}, $\{a_k u_k\}_{k=1}^{r^2}$ is an $\cO_\infty$-basis of $\fO_\infty$, and each $a_k u_k$ lies in $\fP_\infty^{j_k} \setminus \fP_\infty^{j_k+1}$. By Nakayama, the images $\overline{a_k u_k}$ form an $\F_q$-basis of $\fO_\infty/\varpi_\infty \fO_\infty$ adapted to the filtration. Hence each value $j \in \{0, 1, \ldots, r-1\}$ occurs exactly $r$ times in the multiset $\{j_k\}$; otherwise some graded piece would be over- or under-spanned.
		
		Since $j_k = (-e_k) \bmod r$, the same uniform distribution holds for $\{e_k \bmod r\}$.
	\end{proof}
	
	%%%%%%%%%%%%%%%%%%%%%%%%%%%%%%%%%%%%%%%%%%%%%%%
	\subsection{The sum of the SMB \texorpdfstring{$\tau$}{tau}-degrees}\label{sub:sum-ek}
	%%%%%%%%%%%%%%%%%%%%%%%%%%%%%%%%%%%%%%%%%%%%%%%
	
	The next proposition evaluates $\sum_k e_k$ directly from the discriminant of $\La$, using the trace pairing and Lemma~\ref{lem:scaled-SMB-basis}.
	
	\begin{prop}\label{prop:sum-ek}
		For every lattice $\La$,
		\[
		\sum_{k=1}^{r^2} e_k = \frac{r^2(r-1)d}{2}.
		\]
	\end{prop}
	
	\begin{proof}
		Let $u_1, \ldots, u_{r^2}$ be an SMB of $\La$ and let $a_k \in \Fi^\times$ be as in Lemma~\ref{lem:scaled-SMB-basis}, so $\{a_k u_k\}$ is an $\cO_\infty$-basis of $\fO_\infty$. Consider the symmetric bilinear pairing on $D_\infty$ given by the reduced trace, $\langle x, y\rangle \colonequals \Tr(xy)$, where $\Tr\colon D_\infty \to \Fi$ is the reduced trace.
		
		Form the Gram matrices
		\[
		G_\La \colonequals \bigl(\Tr(u_i u_j)\bigr)_{1\leq i,j\leq r^2}, \qquad G_\infty \colonequals \bigl(\Tr(a_i u_i \cdot a_j u_j)\bigr)_{1\leq i,j\leq r^2}.
		\]
		Let $\delta = \diag(a_1, \ldots, a_{r^2})$. By bilinearity, $G_\infty = \delta\cdot G_\La \cdot \delta$, hence
		\begin{equation}\label{eq:gram-relation}
			\det(G_\infty) = \Bigl(\prod_{k=1}^{r^2} a_k\Bigr)^2 \cdot \det(G_\La).
		\end{equation}
		
		We evaluate each side via the discriminant identity \cite[Thm.~32.1]{Reiner}.
		\begin{itemize}
			\item Since $u_1, \ldots, u_{r^2}$ is an $A$-basis of the maximal order $\La$, $\det(G_\La)$ generates 
			\[\disc(\La/A) = \fp^{r(r-1)}\] as a fractional $A$-ideal, so
			\[
			\abs{\det(G_\La)} = \abs{\fp}^{r(r-1)} = q^{r(r-1)d}.
			\]
			\item Since $a_1 u_1, \ldots, a_{r^2} u_{r^2}$ is an $\cO_\infty$-basis of the maximal $\cO_\infty$-order $\fO_\infty$, $\det(G_\infty)$ generates $\disc(\fO_\infty/\cO_\infty) = \varpi_\infty^{r(r-1)}$, so
			\[
			\abs{\det(G_\infty)} = \abs{\varpi_\infty}^{r(r-1)} = q^{-r(r-1)}.
			\]
		\end{itemize}
		Taking absolute values in \eqref{eq:gram-relation},
		\[
		q^{-r(r-1)} = \Bigl(\prod_k \abs{a_k}\Bigr)^2 \cdot q^{r(r-1)d},
		\]
		hence
		\[
		\prod_k \abs{a_k}^{-1} = q^{r(r-1)(d+1)/2}.
		\]
		Since $\abs{a_k} = q^{-\lceil e_k/r\rceil}$,
		\begin{equation}\label{eq:sum-ceilings}
			\sum_{k=1}^{r^2} \lceil e_k/r\rceil = \frac{r(r-1)(d+1)}{2}.
		\end{equation}
		Finally, write $e_k = r\lfloor e_k/r\rfloor + (e_k \bmod r)$, so
		\[
		\lceil e_k/r\rceil = \frac{e_k}{r} + \frac{r - (e_k \bmod r)}{r}\cdot \mathbf{1}_{e_k \bmod r \neq 0}.
		\]
		By Proposition~\ref{prop:uniform-residues}, each residue class $j \in \{1, \ldots, r-1\}$ appears exactly $r$ times among $\{e_k \bmod r\}$, so
		\begin{equation}\label{eq:sum-ceiling-vs-sum}
			\sum_k \lceil e_k/r\rceil = \frac{1}{r}\sum_k e_k + r\sum_{j=1}^{r-1}\frac{r-j}{r} = \frac{1}{r}\sum_k e_k + \frac{r(r-1)}{2}.
		\end{equation}
		Combining \eqref{eq:sum-ceilings} and \eqref{eq:sum-ceiling-vs-sum},
		\[
		\sum_k e_k = r\Bigl(\frac{r(r-1)(d+1)}{2} - \frac{r(r-1)}{2}\Bigr) = \frac{r^2(r-1)d}{2}. \qedhere
		\]
	\end{proof}
	
	\begin{rem}\label{rmk:sum-ek-invariant}
		Although the SMB multiset $\{e_k\}$ generally depends on $\La$, Proposition~\ref{prop:sum-ek} shows that the sum $\sum_k e_k$ depends only on the algebra $D$, through its discriminant $\fp^{r(r-1)}$. The individual $e_k$ and $e_{\max}(\La)$ can vary; see Remark~\ref{rmk:emax-varies} below.
	\end{rem}
	
	%%%%%%%%%%%%%%%%%%%%%%%%%%%%%%%%%%%%%%%%%%%%%%%
	\subsection{The dimension formula and stabilization}\label{sub:dim-formula}
	%%%%%%%%%%%%%%%%%%%%%%%%%%%%%%%%%%%%%%%%%%%%%%%
	
	We now turn to the count of $\cM_s$. 
	%The key point is that norm-orthogonality \eqref{eq:norm-orthog} converts the condition $\deg_\tau(u)\leq s$ into independent conditions on the coordinates of $u$ in an SMB.
	
	\begin{thm}\label{thm:dimMs-elementary}
		Let $\La$ be a lattice with SMB $\tau$-degrees $\{e_k\}_{k=1}^{r^2}$. For every integer $s\geq 0$,
		\begin{equation}\label{eq:dimMs-elementary}
			\dim_{\F_q}\cM_s = \sum_{k=1}^{r^2} \max\!\left(0,\; \left\lfloor \frac{s - e_k}{r}\right\rfloor + 1\right).
		\end{equation}
	\end{thm}
	
	\begin{proof}
		Let $u_1, \ldots, u_{r^2}$ be an SMB of $\La$. Every $u\in \La$ has a unique expression
		\[
		u = \sum_{k=1}^{r^2} \phi_{a_k}(u_k), \qquad a_k \in A,
		\]
		where $\phi_{a_k}(u_k) = a_k \circ u_k$ denotes the $A$-action on $\La$. After $\otimes_A \Fi$, this $A$-action corresponds to left multiplication by $a_k \in A \subset F \subset \Fi$ inside $D_\infty$. Norm-orthogonality \eqref{eq:norm-orthog} with $c_k = a_k$ gives
		\[
		\norm{u}_D = \max_k \abs{a_k} \cdot \norm{u_k}_D = \max_k q^{\deg(a_k) + e_k/r},
		\]
		with the convention $\deg(0) = -\infty$. Hence $u \in \cM_s$, equivalently $\norm{u}_D \leq q^{s/r}$, holds if and only if
		\[
		\deg(a_k) \leq \frac{s - e_k}{r} \quad \text{for all } k.
		\]
		Since $\deg(a_k) \in \Z \cup \{-\infty\}$, this is equivalent to $\deg(a_k) \leq \lfloor(s-e_k)/r\rfloor$.
		
		The conditions on the $a_k$ are independent, so we obtain an $\F_q$-linear direct sum decomposition
		\[
		\cM_s = \bigoplus_{k=1}^{r^2} \bigl\{\phi_{a_k}(u_k) : a_k \in A,\ \deg(a_k) \leq N_k\bigr\}, \qquad N_k \colonequals \lfloor (s-e_k)/r\rfloor.
		\]
		The $k$-th summand is the space of polynomials in $A = \F_q[T]$ of degree at most $N_k$ (acting on $u_k$), which has $\F_q$-dimension $\max(0, N_k + 1)$. Summing over $k$ yields \eqref{eq:dimMs-elementary}.
	\end{proof}
	
	\begin{rem}
		The content of Theorem~\ref{thm:dimMs-elementary} is that norm-orthogonality turns the count of $\cM_s$ into a sum of independent one-dimensional counts. For an arbitrary $A$-basis of $\La$ one obtains only the lower bound $\dim_{\F_q}\cM_s \geq \sum_k \max(0, \lfloor(s-e_k)/r\rfloor + 1)$, since leading-degree cancellation among basis vectors can occur; with an SMB, no such cancellation is possible and the bound becomes an equality.
	\end{rem}
	
	To exploit \eqref{eq:dimMs-elementary} we record an algebraic identity satisfied by the multiset $\{e_k\}$. It depends only on the two structural facts established in Sections~\ref{sub:smb-unit-ball} and \ref{sub:sum-ek}: residue equidistribution (Proposition~\ref{prop:uniform-residues}) and the value $\sum_k e_k = r^2(r-1)d/2$ (Proposition~\ref{prop:sum-ek}).
	
	\begin{lem}\label{lem:floor-identity}
		For every integer $s \geq 0$,
		\[
		\sum_{k=1}^{r^2}\left(\left\lfloor\frac{s-e_k}{r}\right\rfloor + 1\right) = r(s+1) - \frac{r(r-1)(d-1)}{2}.
		\]
	\end{lem}
	
	\begin{proof}
		Write $e_k = r\lfloor e_k/r\rfloor + (e_k \bmod r)$. Then
		\[
		\lfloor(s-e_k)/r\rfloor = -\lfloor e_k/r\rfloor + \lfloor(s - (e_k \bmod r))/r\rfloor,
		\]
		where $\lfloor e_k/r\rfloor = (e_k - (e_k \bmod r))/r$ since $0 \leq e_k \bmod r < r$. Summing over $k$ and using Proposition~\ref{prop:uniform-residues} (each residue $\rho \in \{0, \ldots, r-1\}$ occurs exactly $r$ times among $\{e_k \bmod r\}$),
		\[
		\sum_k \lfloor(s-e_k)/r\rfloor = -\frac{1}{r}\sum_k\bigl(e_k - (e_k \bmod r)\bigr) + r\sum_{\rho=0}^{r-1}\lfloor(s-\rho)/r\rfloor.
		\]
		Writing $s = rm + \sigma$ with $\sigma \in \{0, \ldots, r-1\}$, a direct count gives \[\sum_{\rho=0}^{r-1}\lfloor(s-\rho)/r\rfloor = (\sigma+1)m + (r-\sigma-1)(m-1) = s - r + 1.\] Hence
		\[
		\sum_k \lfloor(s-e_k)/r\rfloor = -\frac{1}{r}\Bigl(\sum_k e_k - \frac{r^2(r-1)}{2}\Bigr) + r(s-r+1),
		\]
		where we used $\sum_k(e_k \bmod r) = r\cdot\frac{r(r-1)}{2} = \frac{r^2(r-1)}{2}$. Substituting $\sum_k e_k = r^2(r-1)d/2$ from Proposition~\ref{prop:sum-ek},
		\begin{align*}
			\sum_k \lfloor(s-e_k)/r\rfloor &= -\frac{r(r-1)d}{2} + \frac{r(r-1)}{2} + r(s-r+1) \\
			&= rs - r^2 + r - \frac{r(r-1)(d-1)}{2}.
		\end{align*}
		Adding $r^2$ to both sides yields the claim.
	\end{proof}
	
	The right-hand side of \eqref{eq:dimMs-elementary} stabilizes as soon as $s$ exceeds the largest SMB minimum:
	
	\begin{cor}\label{cor:stable-range-per-Lambda}
		For $s \geq e_{\max}(\La) - r + 1$,
		\[
		\dim_{\F_q}\cM_s = r(s+1) - \frac{r(r-1)(d-1)}{2}.
		\]
	\end{cor}
	
	\begin{proof}
		If $s \geq e_{\max}(\La) - r + 1$, then $\lfloor(s-e_k)/r\rfloor + 1 \geq 1 > 0$ for every $k$, so the maximum in \eqref{eq:dimMs-elementary} is non-binding and $\dim_{\F_q}\cM_s = \sum_k\bigl(\lfloor(s-e_k)/r\rfloor + 1\bigr)$. Apply Lemma~\ref{lem:floor-identity}.
	\end{proof}
	
	The threshold in Corollary~\ref{cor:stable-range-per-Lambda} depends on $\La$ through $e_{\max}(\La)$. We expect a uniform bound:
	
	\begin{conj}\label{conj:emax-bound}
		For every lattice $\La$ (i.e., every maximal $A$-order $\La$ in $D$, and every bimodule $\Hom(\phi,\psi)$), $e_{\max}(\La) \leq (r-1)d$.
	\end{conj}
	
	To proceed we rewrite the identity \eqref{eq:dimMs-elementary} as the stable-range value plus a non-negative correction.
	
	\begin{prop}\label{prop:dimMs-additive}
		For every lattice $\La$ and every integer $s \geq 0$,
		\begin{equation}\label{eq:dimMs-additive}
			\dim_{\F_q}\cM_s = r(s+1) - \frac{r(r-1)(d-1)}{2} + \sum_{k=1}^{r^2}\max\!\left(0,\, \left\lceil\frac{e_k - s}{r}\right\rceil - 1\right).
		\end{equation}
	\end{prop}
	
	\begin{proof}
		By Lemma~\ref{lem:floor-identity},
		\[
		\sum_{k=1}^{r^2}\left(\lfloor(s-e_k)/r\rfloor + 1\right) = r(s+1) - \frac{r(r-1)(d-1)}{2}.
		\]
		For any integer $x$, $\max(0, x) = x + \max(0, -x)$. Apply this with $x = \lfloor(s-e_k)/r\rfloor + 1$, noting that
		\[
		-x = -\lfloor(s-e_k)/r\rfloor - 1 = \lceil(e_k - s)/r\rceil - 1.
		\]
		Summing over $k$ in \eqref{eq:dimMs-elementary},
		\begin{align*}
			\dim_{\F_q}\cM_s &= \sum_k\max\!\left(0,\, \lfloor(s-e_k)/r\rfloor + 1\right) \\
			&= \sum_k\left(\lfloor(s-e_k)/r\rfloor + 1\right) + \sum_k\max\!\left(0,\, \lceil(e_k - s)/r\rceil - 1\right) \\
			&= r(s+1) - \frac{r(r-1)(d-1)}{2} + \sum_k\max\!\left(0,\, \lceil(e_k - s)/r\rceil - 1\right). \qedhere
		\end{align*}
	\end{proof}
	
	\begin{rem}
		The correction term $\sum_k\max(0,\lceil(e_k-s)/r\rceil - 1)$ in \eqref{eq:dimMs-additive} is non-negative, and the $k$-th summand vanishes if and only if $e_k \leq s + r$. Hence the entire correction vanishes if and only if $s \geq e_{\max}(\La) - r$, which recovers Corollary~\ref{cor:stable-range-per-Lambda} and slightly sharpens it: the stable formula already holds at $s = e_{\max}(\La) - r$, when this is non-negative.
	\end{rem}
	
	The boundary value $s = 0$ of \eqref{eq:dimMs-elementary} extracts the number of vanishing SMB minima.
	
	\begin{prop}\label{prop:m0-count}
		For every lattice $\La$,
		\[
		\dim_{\F_q}\cM_0 = \#\{k : e_k = 0\}.
		\]
	\end{prop}
	
	\begin{proof}
		By \eqref{eq:dimMs-elementary} with $s = 0$, the $k$-th summand is $\max(0, \lfloor -e_k/r\rfloor + 1)$. Since $e_k \geq 0$, this equals $1$ when $e_k = 0$ and $0$ when $e_k \geq 1$ (as $\lfloor -e_k/r\rfloor \leq -1$ then). Summing gives the claim.
	\end{proof}
	
	When $\La = \End(\phi)$, the space $\cM_0$ identifies with the constant subring $\End(\phi)\cap \oF_\fp$, whose size is governed by the symmetry of $\phi$ in the following sense.
	
	\begin{lem}\label{lem:e1-zero}
		Let $\phi$ be a Drinfeld module of rank $r$ over $\oF_\fp$ with $\phi_T = t + g_1\tau + \cdots + g_r\tau^r$, and let $\La = \End(\phi)$. Set
		\[
		\ell(\phi) \colonequals \gcd\bigl\{i : 1 \leq i \leq r,\ g_i \neq 0\bigr\},
		\]
		a positive divisor of $r$ (the gcd is well-defined since $g_r \neq 0$). Then:
		\begin{enumerate}
			\item[\textup{(1)}] $\End(\phi) \cap \oF_\fp = \F_{q^{\ell(\phi)}}$ and $\dim_{\F_q}\cM_0 = \ell(\phi)$;
			\item[\textup{(2)}] $\#\{k : e_k = 0\} = \ell(\phi)$, i.e., $e_1 = \cdots = e_{\ell(\phi)} = 0$ and $e_{\ell(\phi)+1} \geq 1$;
			\item[\textup{(3)}] if $r$ is prime, then $\ell(\phi) \in \{1, r\}$: either $\ell(\phi) = 1$ and $e_1 = 0,\, e_2 \geq 1$, or $\phi_T = t + g_r\tau^r$ and $e_1 = \cdots = e_r = 0$.
		\end{enumerate}
	\end{lem}
	
	\begin{proof}
		An element $\alpha \in \oF_\fp$ lies in $\End(\phi)$ if and only if $\alpha \phi_T = \phi_T \alpha$ in $\oF_\fp\{\tau\}$. Equating coefficients of $\tau^i$, this is the system $\alpha g_i = g_i \alpha^{q^i}$ for $0 \leq i \leq r$ (with $g_0 = t$). For $i = 0$ it is vacuous; for $i \geq 1$ with $g_i \neq 0$ it forces $\alpha = \alpha^{q^i}$, i.e., $\alpha \in \F_{q^i}$; for $i \geq 1$ with $g_i = 0$ there is no constraint. The set of such $\alpha$ is therefore $\bigcap_{g_i \neq 0}\F_{q^i} = \F_{q^{\ell(\phi)}}$, proving (1), since $\cM_0 = \End(\phi) \cap \oF_\fp$ (elements of $\tau$-degree $0$ are precisely the constants). Claim (2) is immediate from (1) and Proposition~\ref{prop:m0-count}. For (3), $\ell(\phi)$ divides $r$ by definition; when $r$ is prime its only divisors are $1$ and $r$, and $\ell(\phi) = r$ means $g_i = 0$ for $1 \leq i \leq r-1$, i.e., $\phi_T = t + g_r\tau^r$.
	\end{proof}
	
	\begin{rem}\label{rmk:hom-e1}
		For $\La = \Hom(\phi, \psi)$ with $\phi \not\cong \psi$, there is no canonical element analogous to $1 \in \End(\phi)$, so the constraint $e_1 = 0$ need not hold. In this case $e_1$ equals the minimum $\tau$-degree of a nonzero isogeny $\phi \to \psi$, a nontrivial invariant of the pair $(\phi,\psi)$. It would be interesting to determine which multisets $\{e_k\}$ actually arise for $\Hom(\phi,\psi)$ beyond the constraints of this section.
	\end{rem}
	
	Proposition~\ref{prop:dimMs-additive} converts the conjectured stabilization range into a purely numerical condition on $\{e_k\}$.
	
	\begin{cor}\label{cor:conj-equivalence}
		For every lattice $\La$, the following are equivalent:
		\begin{enumerate}
			\item[\textup{(1)}] $e_{\max}(\La) \leq (r-1)d$;
			\item[\textup{(2)}] $\dim_{\F_q}\cM_s = r(s+1) - \dfrac{r(r-1)(d-1)}{2}$ for all $s \geq (r-1)(d-1) - 1$;
			\item[\textup{(3)}] $\dim_{\F_q}\cM_{(r-1)(d-1)-1} = \dfrac{r(r-1)(d-1)}{2}$.
		\end{enumerate}
	\end{cor}
	
	\begin{proof}
		By Proposition~\ref{prop:dimMs-additive}, the difference between $\dim_{\F_q}\cM_s$ and the stable-range value is the non-negative correction $C(s) \colonequals \sum_k\max(0, \lceil(e_k-s)/r\rceil - 1)$, which vanishes if and only if $e_k \leq s + r$ for all $k$, i.e., $s \geq e_{\max}(\La) - r$.
		
		(1) $\Leftrightarrow$ (2): if (1) holds, then for every $s \geq (r-1)(d-1) - 1$,
		\[
		s + r \geq (r-1)(d-1) - 1 + r = (r-1)d \geq e_{\max}(\La),
		\]
		so $C(s) = 0$ and (2) follows. Conversely, if (2) holds, then $C((r-1)(d-1) - 1) = 0$, so $e_{\max}(\La) - r \leq (r-1)(d-1) - 1$, equivalently $e_{\max}(\La) \leq (r-1)d$.
		
		(2) $\Rightarrow$ (3) is trivial; the converse is the case $s = (r-1)(d-1) - 1$ of the equivalence above.
	\end{proof}
	
	\begin{thm}\label{thm:main-stabilization}
		If Conjecture~\ref{conj:emax-bound} holds for $\La$, then
		\[
		\dim_{\F_q}\cM_s = r(s+1) - \frac{r(r-1)(d-1)}{2} \qquad \text{for all } s \geq (r-1)(d-1) - 1.
		\]
		Conversely, the validity of this stabilization formula for the single value $s = (r-1)(d-1) - 1$ implies Conjecture~\ref{conj:emax-bound} for $\La$.
	\end{thm}
	
	\begin{proof}
		Combine (1)$\Leftrightarrow$(2)$\Leftrightarrow$(3) of Corollary~\ref{cor:conj-equivalence}.
	\end{proof}
	
	\begin{cor}\label{cor:emax-r2}
		Conjecture~\ref{conj:emax-bound} holds for every lattice $\La$ of rank $r = 2$. Equivalently, for every supersingular pair $\phi, \psi$ of rank $2$ over $\oF_\fp$ and every SMB of $\Hom(\phi,\psi)$, one has $e_{\max} \leq d$.
	\end{cor}
	
	\begin{proof}
		Theorem~\ref{thmMain}, proved independently in Section~\ref{secBrandt}, states that for $r = 2$,
		\[
		\dim_{\F_q}\cM_s(\phi,\psi) = 2(s+1) - (d-1) \qquad \text{for all } s \geq d - 2 = (r-1)(d-1) - 1.
		\]
		This is condition (2) of Corollary~\ref{cor:conj-equivalence} for $r = 2$.
	\end{proof}
	
	For $r \geq 3$, Conjecture~\ref{conj:emax-bound} remains open. Nonetheless the constraints already established (residue equidistribution and the value of $\sum_k e_k$) suffice to give an unconditional, if non-optimal, bound on $e_{\max}(\La)$, and hence an unconditional stabilization range.
	
	\begin{prop}\label{prop:emax-crude}
		For every lattice $\La$,
		\[
		e_{\max}(\La) \leq \frac{r^2(r-1)(d-1)}{2} + (r-1).
		\]
		Consequently,
		\[
		\dim_{\F_q}\cM_s = r(s+1) - \frac{r(r-1)(d-1)}{2} \qquad \text{for all }s \geq \frac{r^2(r-1)(d-1)}{2}.
		\]
	\end{prop}
	
	\begin{proof}
		Write $e_{\max} = e_{\max}(\La)$ and let $\rho^\ast \in \{0, 1, \ldots, r-1\}$ be its residue mod $r$. By Proposition~\ref{prop:uniform-residues}, each residue class $\rho \in \{0, 1, \ldots, r-1\}$ is represented by exactly $r$ entries among $\{e_k\}$. Among the $r^2 - 1$ entries other than $e_{\max}$, the $r$ entries in each class $\rho \neq \rho^\ast$ are $\geq \rho$ (the smallest non-negative representative), and the $r-1$ remaining entries in class $\rho^\ast$ are $\geq \rho^\ast$. Therefore
		\[
		\sum_{k=1}^{r^2-1} e_k \geq (r-1)\rho^\ast + \sum_{\substack{\rho = 0 \\ \rho \neq \rho^\ast}}^{r-1} r\rho = -\rho^\ast + r\sum_{\rho=0}^{r-1}\rho = \frac{r^2(r-1)}{2} - \rho^\ast.
		\]
		Combining with $\sum_k e_k = r^2(r-1)d/2$ from Proposition~\ref{prop:sum-ek},
		\[
		e_{\max} = \sum_k e_k - \sum_{k=1}^{r^2-1}e_k \leq \frac{r^2(r-1)d}{2} - \frac{r^2(r-1)}{2} + \rho^\ast = \frac{r^2(r-1)(d-1)}{2} + \rho^\ast.
		\]
		Since $\rho^\ast \leq r-1$, this gives the stated bound. The stabilization range now follows from Corollary~\ref{cor:stable-range-per-Lambda}, since $s \geq e_{\max} - r + 1$ holds whenever $s \geq \frac{r^2(r-1)(d-1)}{2}$.
	\end{proof}
	
	\begin{rem}\label{rmk:emax-sharp}
		The bound of Conjecture~\ref{conj:emax-bound}, when it holds, is sharp: for $\La = \End(\phi)$ with $\phi_T = t + \tau^r$ (supersingular precisely when $\gcd(r,d) = 1$), one has equality $e_{\max}(\La) = (r-1)d$. This is established by an explicit successive-minima basis in Proposition~\ref{prop:phi-vanishing}.
	\end{rem}
	
	\begin{rem}\label{rmk:emax-varies}
		The full multiset $\{e_k\}$, and in particular $e_{\max}(\La)$, genuinely depends on $\La$. We illustrate the typical landscape in rank $2$ in Section~\ref{ssCompBM}: the constraints of this section---the value of $\sum_k e_k$, residue equidistribution, the bound $e_{\max}\leq d$ from Corollary~\ref{cor:emax-r2}, and $e_1 = 0$ from Lemma~\ref{lem:e1-zero}---enumerate all SMB multisets that can possibly arise for $\End(\phi)$, and we will see that each of them does.
	\end{rem}
	%==================================================================
	
	%==================================================================
	\section{Mornev's volume and the geometric perspective}\label{sec:mornev-volume}
	
	In this section we reinterpret the results of Section~\ref{sec:smb-stabilization} through Mornev's lattice volume \cite{Mornev}. Nothing here is needed for the stabilization theorem; the purpose is rather to place that theorem in a geometric framework, in which the dimension formula \eqref{eq:dimMs-elementary} acquires the form
	\[
	\dim_{\F_q}\cM_s = rs - \log_q\Vol(\La) + h^1(\cE_{i_s}(j_s\infty)),
	\]
	a universal main term $rs - \log_q\Vol(\La)$ governed by the volume of $\La$, plus a cohomological correction $h^1(\cE_{i_s}(j_s\infty))$ that vanishes in the stable range. 
	%From this vantage point the three numerical invariants of $(\La, \norm{\cdot}_D)$ that have appeared---Mornev's volume $\Vol(\La)$, the Taguchi discriminant $\Delta(\La)$, and the SMB multiset $\{e_k\}$---are three repackagings of the same data, tied together by the residue equidistribution of Proposition~\ref{prop:uniform-residues}. 
	The geometric picture also recasts Conjecture~\ref{conj:emax-bound} as a cohomology-vanishing statement (Corollary~\ref{cor:cohom-vanishing}).
	
	Throughout, $\cO_{F,\infty} \subset F$ denotes the local ring at $\infty \in \PP^1_{\F_q}$, i.e., the rational functions regular at $\infty$, whose completion is the ring of integers $\cO_\infty$ of $\Fi$ introduced in Section~\ref{sub:lattice-structure}.
	
	%%%%%%%%%%%%%%%%%%%%%%%%%%%%%%%%%%%%%%%%%%%%%%%
	\subsection{Mornev's volume and the count formula}\label{sub:mornev-formula}
	%%%%%%%%%%%%%%%%%%%%%%%%%%%%%%%%%%%%%%%%%%%%%%%
	
	We recall the volume of an $A$-lattice from \cite[\S 1.5]{Mornev}. For each real $\alpha > 0$, let $\mu_\alpha$ be the Haar measure on $V = D_\infty$ normalized so that $\mu_\alpha(B(V,\alpha)) = 1$, where $B(V,\alpha) = \{x \in V : \norm{x}_D \leq \alpha\}$. The $\alpha$-normalized volume of $(\La, \norm{\cdot}_D)$ is
	\[
	\Vol(\La, \alpha) \colonequals \mu_\alpha(V/\La), \qquad \Vol(\La) \colonequals \Vol(\La, 1).
	\]
	With this normalization, $\mu_1(\fO_\infty)=1$.
	
	The norm $\norm{\cdot}_D$ takes values in $q^{(1/r)\Z}$, finer than the value group $q^\Z$ of $\abs{\cdot}_\infty$. Accordingly, by \cite[Cor.~1.4.3]{Mornev}, the lattice $(\La, \norm{\cdot}_D)$ corresponds to a strictly increasing chain of locally free sheaves on $X = \PP^1_{\F_q}$ together with a sequence of reals in $(q^{-1}, 1]$. In our setting the $r$ residues are
	\[
	\alpha_i \colonequals q^{(i-r)/r}, \qquad i = 1, 2, \ldots, r,
	\]
	and the chain consists of $r$ rank-$r^2$ vector bundles
	\[
	\cE_1 \subset \cE_2 \subset \cdots \subset \cE_r \subset \cE_1(\infty)
	\]
	on $X$, all agreeing on the open set $\A^1_{\F_q} = \Spec(A)$. Concretely, $\cE_i$ is obtained by gluing the locally free sheaf on $\A^1_{\F_q}$ induced by $\La$ to the rational ball
	\[
	B(D, \alpha_i) \colonequals \{x \in D : \norm{x}_D \leq \alpha_i\},
	\]
	a finitely generated $\cO_{F,\infty}$-submodule of $D$ by \cite[Cor.~1.3.4]{Mornev}, via the canonical isomorphism
	\[
	F \otimes_A \La \xrightarrow{\sim} D \xrightarrow{\sim} F \otimes_{\cO_{F,\infty}} B(D, \alpha_i).
	\]
	By \cite[Thm.~1.4.2(2)]{Mornev}, for each $i \in \{1, \ldots, r\}$ and integer $j$,
	\[
	H^0(X, \cE_i(j\infty)) = \{u \in \La : \norm{u}_D \leq \alpha_i q^j\}.
	\]
	We write $h^1(\cE_i(j\infty)) \colonequals \dim_{\F_q} H^1(X, \cE_i(j\infty))$.
	
	\begin{thm}[Mornev's count formula]\label{thm:mornev-count}
		Fix $i \in \{1, \ldots, r\}$ and an integer $j$. Then
		\[
		\#\{u \in \La : \norm{u}_D \leq \alpha_i q^j\} = \frac{q^{r^2 j}\cdot \#H^1(X, \cE_i(j\infty))}{\Vol(\La, \alpha_i)}.
		\]
	\end{thm}
	
	\begin{proof}
		The proof of \cite[Thm.~1.5.4]{Mornev} computes the cohomology of $\cE_i(j\infty)$ via a \v{C}ech complex concentrated in degrees $0$ and $1$, yielding for any translation-invariant Haar measure $\mu$ on $V$ the identity
		\[
		\frac{\mu(V/\La)}{\mu(B(V, \alpha_i q^j))} = \frac{\#H^1(X, \cE_i(j\infty))}{\#H^0(X, \cE_i(j\infty))}.
		\]
		Taking $\mu = \mu_{\alpha_i q^j}$, the left-hand side equals $\Vol(\La, \alpha_i q^j) = q^{-r^2 j}\Vol(\La, \alpha_i)$. Rearranging and using $H^0(X, \cE_i(j\infty)) = \{u \in \La : \norm{u}_D \leq \alpha_i q^j\}$ yields the claim.
	\end{proof}
	
	To convert Theorem~\ref{thm:mornev-count} into a count of $\cM_s$, we relate the slicing parameters $(i, j)$ to $s$.
	
	\begin{defn}\label{def:js-is}
		For each integer $s \geq 0$, write $s = rj_s + (i_s - r)$ with $j_s = \lceil s/r\rceil$ and $i_s \in \{1, \ldots, r\}$ determined by $i_s \equiv s \pmod r$ (with $i_s = r$ when $r \mid s$). Equivalently, $q^{s/r} = \alpha_{i_s} q^{j_s}$, so
		\[
		\cM_s = \{u \in \La : \deg_\tau(u) \leq s\} = \{u \in \La : \norm{u}_D \leq q^{s/r}\} = H^0(X, \cE_{i_s}(j_s \infty)).
		\]
	\end{defn}
	
	A short calculation, replacing $\mu_1(B(V, \alpha_i))$ by $q^{-r(r-i)}$ (since $B(V, \alpha_i) = \fP_\infty^{r-i}$ has index $q^{r(r-i)}$ in $\fO_\infty$), gives $\Vol(\La, \alpha_i) = q^{r(r-i)}\cdot \Vol(\La)$. Theorem~\ref{thm:mornev-count} then yields:
	
	\begin{cor}\label{cor:dimMs-mornev}
		For every integer $s \geq 0$,
		\begin{equation}\label{eq:dimMs-mornev}
			\dim_{\F_q}\cM_s = rs - \log_q\Vol(\La) + h^1(\cE_{i_s}(j_s\infty)).
		\end{equation}
		The correction term $h^1(\cE_{i_s}(j_s\infty))$ is a non-negative integer that vanishes for all $s \gg 0$.
	\end{cor}
	
	\begin{proof}
		Theorem~\ref{thm:mornev-count} gives $\dim_{\F_q}\cM_s = r^2 j_s - \log_q\Vol(\La, \alpha_{i_s}) + h^1(\cE_{i_s}(j_s\infty))$. Substituting $\Vol(\La, \alpha_{i_s}) = q^{r(r-i_s)}\Vol(\La)$ and using $rj_s + i_s - r = s$,
		\begin{align*}
			\dim_{\F_q}\cM_s &= r^2 j_s - r(r - i_s) - \log_q\Vol(\La) + h^1(\cE_{i_s}(j_s\infty)) \\
			&= rs - \log_q\Vol(\La) + h^1(\cE_{i_s}(j_s\infty)).
		\end{align*}
		The eventual vanishing of $h^1$ follows from Serre vanishing, since $\cO_X(\infty)$ is ample.
	\end{proof}
	
	Corollary~\ref{cor:dimMs-mornev} exhibits $\dim_{\F_q}\cM_s$ as a universal main term $rs - \log_q\Vol(\La)$, depending only on the volume of $\La$, plus an eventually-vanishing cohomological correction. We treat each in turn: the volume in Section~\ref{sub:vol-via-smb}, the correction in Section~\ref{sub:vanishing}.
	
	%%%%%%%%%%%%%%%%%%%%%%%%%%%%%%%%%%%%%%%%%%%%%%%
	\subsection{The volume of \texorpdfstring{$\La$}{Lambda} via successive minima}\label{sub:vol-via-smb}
	%%%%%%%%%%%%%%%%%%%%%%%%%%%%%%%%%%%%%%%%%%%%%%%
	
	We compute $\Vol(\La)$ directly from the SMB data of Section~\ref{sub:smb-unit-ball}.
	
	\begin{prop}\label{prop:vol-via-ceilings}
		For every lattice $\La$,
		\[
		\log_q \Vol(\La) = -r^2 + \sum_{k=1}^{r^2} \lceil e_k/r \rceil.
		\]
	\end{prop}
	
	\begin{proof}
		Mornev's determinant formula \cite[Lem.~1.5.3]{Mornev} specializes in our setting (curve $X = \PP^1_{\F_q}$, genus $0$, place $\infty$ of degree $1$, lattice rank $r^2$) to
		\begin{equation}\label{eq:vol-via-omega}
			\Vol(\La) = q^{-r^2}\cdot \abs{\frac{\Omega_\La}{\Omega_\infty}}_\infty,
		\end{equation}
		for any $A$-basis $\la_1, \ldots, \la_{r^2}$ of $\La$ and any $\cO_\infty$-basis $v_1, \ldots, v_{r^2}$ of $\fO_\infty$, where $\Omega_\La \colonequals \la_1 \wedge \cdots \wedge \la_{r^2}$ and $\Omega_\infty \colonequals v_1 \wedge \cdots \wedge v_{r^2}$, both viewed in the one-dimensional $\Fi$-vector space $\det_{\Fi} D_\infty$.
		
		Apply this with $\la_k = u_k$ an SMB of $\La$ and $v_k = a_k u_k$ from Lemma~\ref{lem:scaled-SMB-basis}. Then
		\[
		\frac{u_1 \wedge \cdots \wedge u_{r^2}}{(a_1 u_1) \wedge \cdots \wedge (a_{r^2}u_{r^2})} = \prod_{k=1}^{r^2} a_k^{-1},
		\]
		so $\abs{\Omega_\La/\Omega_\infty}_\infty = \prod_k \abs{a_k}_\infty^{-1} = q^{\sum_k \lceil e_k/r\rceil}$. Substituting into \eqref{eq:vol-via-omega} gives the claim.
	\end{proof}
	
	\begin{cor}\label{cor:vol-explicit}
		For every lattice $\La$,
		\[
		\log_q\Vol(\La) = \frac{r(r-1)(d-1)}{2} - r.
		\]
	\end{cor}
	
	\begin{proof}
		By Proposition~\ref{prop:vol-via-ceilings} and \eqref{eq:sum-ceilings} from the proof of Proposition~\ref{prop:sum-ek},
		\begin{align*}
		\log_q \Vol(\La) &= -r^2 + \sum_k \lceil e_k/r\rceil \\
		&= -r^2 + \frac{r(r-1)(d+1)}{2} \\
		&= \frac{r(r-1)(d-1)}{2} - r. \qedhere
		\end{align*}
	\end{proof}
	
	\begin{defn} The \textit{Taguchi discriminant} of $\La$ \cite[\S 4]{Taguchi} is 
	\[
	\Delta(\La) \colonequals \prod_{k=1}^{r^2}\norm{u_k}_D = q^{(1/r)\sum_k e_k},
	\]
	where $u_1, \ldots, u_{r^2}$ is any SMB of $\La$. By Proposition~\ref{prop:sum-ek}, $\log_q\Delta(\La) = r(r-1)d/2$. 
	\end{defn}
	%; by \cite[Lem.~4.2]{Taguchi} it is independent of the SMB. 
	The volume and the discriminant differ by a universal constant:
	
	\begin{thm}\label{thm:tag-vs-mor}
		For every lattice $\La$,
		\[
		\log_q\Delta(\La) - \log_q \Vol(\La) = \frac{r(r+1)}{2}.
		\]
	\end{thm}
	
	\begin{proof}
		Subtract Proposition~\ref{prop:vol-via-ceilings} from $\log_q \Delta(\La) = \frac{1}{r}\sum_k e_k$ and apply \eqref{eq:sum-ceiling-vs-sum} from the proof of Proposition~\ref{prop:sum-ek}:
		\begin{align*}
			\log_q \Delta(\La) - \log_q \Vol(\La) &= \frac{1}{r}\sum_k e_k + r^2 - \sum_k \lceil e_k/r\rceil \\
			&= r^2 - \frac{r(r-1)}{2} = \frac{r(r+1)}{2}. \qedhere
		\end{align*}
	\end{proof}
	
	%%%%%%%%%%%%%%%%%%%%%%%%%%%%%%%%%%%%%%%%%%%%%%%
	\subsection{Cohomology vanishing and the geometric stability threshold}\label{sub:vanishing}
	%%%%%%%%%%%%%%%%%%%%%%%%%%%%%%%%%%%%%%%%%%%%%%%
	
	The correction term $h^1(\cE_{i_s}(j_s\infty))$ in Corollary~\ref{cor:dimMs-mornev} can be computed explicitly from the SMB $\tau$-degrees. This yields a geometric reading of the stabilization threshold of Section~\ref{sub:dim-formula}, and recasts Conjecture~\ref{conj:emax-bound} as the vanishing of $H^1$ in the conjectured stable range.
	
	By Grothendieck's theorem on $X = \PP^1_{\F_q}$, each bundle $\cE_i$ splits as $\cE_i = \bigoplus_{k=1}^{r^2}\cO(d_{i,k})$, where $\cO(d) \colonequals \cO_X(d\cdot \infty)$. The splitting is governed by the SMB.
	
	\begin{prop}\label{prop:bundle-splitting}
		For each $i \in \{1, \ldots, r\}$,
		\[
		\cE_i = \bigoplus_{k=1}^{r^2} \cO(d_{i,k}), \qquad d_{i,k} = -\lceil (e_k + r - i)/r\rceil.
		\]
	\end{prop}
	
	\begin{proof}
		For $i = r$, the bundle $\cE_r$ glues $\La$ on $\A^1_{\F_q}$ to the rational unit ball $B(D, 1) = \{x \in D : \norm{x}_D \leq 1\}$ at $\infty$. By Lemma~\ref{lem:scaled-SMB-basis}, choosing $a_k \in F^\times$ with $\ord_\infty(a_k) = \lceil e_k/r\rceil$ (possible since $\ord_\infty\colon F^\times \twoheadrightarrow \Z$ is surjective; e.g., $a_k = T^{-\lceil e_k/r\rceil}$), the set $\{a_k u_k\}_{k=1}^{r^2}$ is an $\cO_{F,\infty}$-basis of $B(D, 1)$. Hence $\cE_r$ decomposes into rank-one summands: the $k$-th glues $A u_k$ on $\A^1$ to $\cO_{F,\infty}(a_k u_k)$ at $\infty$ via $u_k = a_k^{-1}(a_k u_k)$. This is precisely $\cO(-\lceil e_k/r\rceil)$, giving $d_{r,k} = -\lceil e_k/r\rceil$.
		
		For general $i$, repeat the argument with the bound $\norm{x}_D \leq \alpha_i = q^{(i-r)/r}$ in place of $\norm{x}_D \leq 1$: the rescaled vectors $\{a_k^{(i)} u_k\}$ with $\ord_\infty(a_k^{(i)}) = \lceil (e_k + r - i)/r\rceil$ form an $\cO_{F,\infty}$-basis of $B(D, \alpha_i)$. Therefore $d_{i,k} = -\lceil (e_k + r - i)/r\rceil$, recovering $d_{r,k} = -\lceil e_k/r\rceil$ when $i = r$.
	\end{proof}
	
	\begin{prop}\label{prop:h1-explicit}
		Let $\La$ be a lattice with SMB $\tau$-degrees $\{e_k\}_{k=1}^{r^2}$. For every integer $s \geq 0$,
		\begin{equation}\label{eq:h1-formula}
			h^1(\cE_{i_s}(j_s\infty)) = \sum_{k=1}^{r^2} \max\!\left(0,\ \left\lceil \frac{e_k - s}{r}\right\rceil - 1\right).
		\end{equation}
		The correction vanishes precisely when $s \geq e_{\max}(\La) - r$.
	\end{prop}
	
	\begin{proof}
		By Proposition~\ref{prop:bundle-splitting}, $\cE_{i_s}(j_s\infty) = \bigoplus_k \cO(d_{i_s,k} + j_s)$. On $\PP^1$,
		\[
		h^1(\cO(d)) = \max(0, -d-1),
		\]
		so
		\[
		h^1(\cE_{i_s}(j_s\infty)) = \sum_k \max\bigl(0,\, \lceil (e_k + r - i_s)/r\rceil - j_s - 1\bigr).
		\]
		For any integer $n$, $\lceil (n+r)/r\rceil = \lceil n/r\rceil + 1$, hence $\lceil (e_k + r - i_s)/r\rceil - 1 = \lceil (e_k - i_s)/r\rceil$. The relation $rj_s + i_s = s + r$ gives
		\[
		\lceil (e_k - i_s)/r\rceil - j_s = \lceil (e_k - i_s - rj_s)/r\rceil = \lceil (e_k - s - r)/r\rceil = \lceil (e_k - s)/r\rceil - 1,
		\]
		which proves \eqref{eq:h1-formula}. The final claim is immediate: $\lceil (e_k - s)/r\rceil \leq 1$ if and only if $e_k \leq s + r$.
	\end{proof}
	
	Comparing \eqref{eq:h1-formula} with the additive form \eqref{eq:dimMs-additive} of the dimension formula, we see that the cohomological correction $h^1(\cE_{i_s}(j_s\infty))$ is exactly the non-negative correction term of Proposition~\ref{prop:dimMs-additive}. Thus the two perspectives agree term by term, and the geometric route reproves the stabilization formula of Section~\ref{sub:dim-formula}: the stable range is precisely the range of cohomology vanishing.
	
	\begin{cor}\label{cor:cohom-vanishing}
		For every lattice $\La$, the following are equivalent:
		\begin{enumerate}
			\item[\textup{(1)}] $e_{\max}(\La) \leq (r-1)d$ \textup{(}Conjecture~\ref{conj:emax-bound}\textup{)};
			\item[\textup{(2)}] $H^1(X, \cE_{i_s}(j_s\infty)) = 0$ for all $s \geq (r-1)(d-1) - 1$.
		\end{enumerate}
		In particular, Conjecture~\ref{conj:emax-bound} is equivalent to the vanishing of $H^1$ throughout the conjectured stable range, and holds unconditionally for $r = 2$ by Corollary~\ref{cor:emax-r2}.
	\end{cor}
	
	\begin{proof}
		By Proposition~\ref{prop:h1-explicit}, $H^1(X, \cE_{i_s}(j_s\infty)) = 0$ if and only if $s \geq e_{\max}(\La) - r$. Hence vanishing holds for all $s \geq (r-1)(d-1) - 1$ if and only if $e_{\max}(\La) - r \leq (r-1)(d-1) - 1$, i.e., $e_{\max}(\La) \leq (r-1)d$.
	\end{proof}

		%==================================================================
	
	%==================================================================
	\section{Rank-2 case: stabilization via theta series}\label{secBrandt}
	
	In this section we prove Theorem~\ref{thmMain}, the explicit rank-$2$ form of the stabilization formula $\dim_{\F_q}\cM_s = 2(s+1)-(d-1)$ for all $s\geq d-2$. The argument is automorphic: we encode the dimensions $\dim_{\F_q}\cM_s$ via a Brandt-type theta series on the Bruhat--Tits tree, decompose it into Eisenstein and cuspidal parts, and use the polynomiality of the cuspidal $L$-function to obtain a recursion that pins down $\dim_{\F_q}\cM_s$ exactly. This route is independent of the successive-minima machinery of Section~\ref{sec:smb-stabilization}, and by Corollary~\ref{cor:emax-r2} it establishes Conjecture~\ref{conj:emax-bound} for $r=2$.
	
	Fix a prime $\fp\lhd A$ of degree $d$ and let $D$ be the quaternion algebra over $F$ ramified only at $\fp$ and $\infty$. Two left ideals $I$ and $J$ of a maximal $A$-order in $D$ are in the same class if $J=Ib$ for some $b\in D^\times$. The set of left ideal classes is finite, and its cardinality $n$ is independent of the choice of maximal order. Fix a maximal order, let $\{I_1, \dots, I_n\}$ be a set of representatives of the distinct ideal classes with $I_1$ equal to the chosen order, and for $1\leq i\leq n$ let $\La_i$ be the right order of $I_i$. Each conjugacy class of maximal orders in $D$ is represented (once or twice) by some $\La_i$; we let $\mathbf{t}\leq n$ denote the number of distinct conjugacy classes. Then $n$ is the \textit{class number} and $\mathbf{t}$ is the \textit{type number} of $D$. With this notation,
	\[
	\deg H_\fp(x) = \begin{cases} n & \text{if $d$ is even,} \\ n-1 & \text{if $d$ is odd,} \end{cases}
	\]
	and the number of irreducible factors of $H_\fp(x)$ in $\F_\fp[x]$ is $\mathbf{t}$ if $d$ is even, $\mathbf{t}-1$ if $d$ is odd, where $H_\fp(x)$ is as in Definition~\ref{defHp}; see also \cite[Prop. 4.6]{Gekeler91}.
	
	By Theorem~\ref{thmGekelerMaxOrder}, the left ideal classes of any maximal $A$-order in $D$ are in bijection with the isomorphism classes of supersingular rank-$2$ Drinfeld modules over $\oF_\fp$. Let $\phi_1, \dots, \phi_n$ be representatives of these isomorphism classes, arranged so that $\La_i = \End(\phi_i)$. For each ordered pair $1\leq i, j\leq n$, set
	\[
	\La_{ij} \colonequals \Hom(\phi_i, \phi_j),
	\]
	an $(\La_j, \La_i)$-bimodule that is free of rank $4$ as an $A$-module on either side. By Section~\ref{sub:lattice-structure}, $\La_{ij}$ is a lattice in the sense of that section, and we write
	\[
	\cM_s^{ij} \colonequals \{u\in \La_{ij} : \deg_\tau(u)\leq s\}, \qquad s\geq 0.
	\]
	The goal of this section is to compute $\dim_{\F_q}\cM_s^{ij}$ for all $s\geq d-2$.
	
	We begin with the Brandt matrices that encode isogeny counts between the $\phi_i$. For monic $\fm\in A_+$, set
	\begin{equation}\label{eqb1}
		b_{ij}(\fm) \colonequals \frac{1}{\#\Aut(\phi_j)} \#\{u\in \La_{ij} : (\fN(u))=(\fm)\},
	\end{equation}
	where $\fN(u)$ is the norm of the isogeny $u$ defined in \eqref{eqNormu}. By Lemma~\ref{lemAut}, $\#\Aut(\phi_j) = q-1$ if $j(\phi_j)\neq 0$ and $q^2-1$ if $j(\phi_j)=0$.
	
	\begin{defn}\label{defnBrandt}
		The $n\times n$ matrix $B(\fm) = (b_{ij}(\fm))_{1\leq i,j\leq n} \in \Mat_n(\Z)$ is the $\fm$-th \textit{Brandt matrix} in $A$-characteristic $\fp$.
	\end{defn}
	
	\begin{rem}
		The integer $b_{ij}(\fm)$ counts isogenies $u\colon\phi_i\to\phi_j$ with $\fN(u)=(\fm)$ up to equivalence, where two isogenies are equivalent if they have the same kernel, i.e., if they differ by an automorphism of $\phi_j$.
	\end{rem}
	
	\begin{rem}\label{rem3.3}
		The lattice $\La_{ij}$ admits a purely quaternion-algebraic description: there is a canonical isomorphism
		\[
		\La_{ij} \cong I_j^{-1}I_i
		\]
		of $(\La_j, \La_i)$-bimodules \cite[Prop. 2.7]{PapikianJNT05}. Let $\fn_{ij}\in F$ be a generator of the reduced norm ideal $\Nr(I_j^{-1}I_i) = \langle \Nr(u) : u\in I_j^{-1}I_i\rangle_A$, where $\Nr\colon D\to F$ is the reduced norm, and set
		\[
		w_j \colonequals \#\Aut(\phi_j)/(q-1) = \#(\La_j^\times)/(q-1).
		\]
		Then $b_{ij}(\fm)$ can be computed from the quaternion data alone as
		\begin{equation}\label{eqb2}
			b_{ij}(\fm) = \frac{\#\{u\in I_j^{-1}I_i : (\Nr(u)/\fn_{ij}) = (\fm)\}}{(q-1)w_j};
		\end{equation}
		the equality of \eqref{eqb1} and \eqref{eqb2} uses \cite[Lem. 3.10]{Gekeler91} and is established in \cite[Prop. 2.7]{PapikianJNT05}.
	\end{rem}
	
	These counts feed into a recursion for $\#\cM_s^{ij}$. By \eqref{eqtaudeg}, $\deg_\tau(u) = \deg_T(\fN(u))$, so
	\[
	\cM_s^{ij} = \bigsqcup_{m=0}^s\bigsqcup_{\substack{\fm\in A_+\\ \deg(\fm)=m}} \{u\in \La_{ij} : (\fN(u)) = (\fm)\}.
	\]
	Counting the inner sets via \eqref{eqb1},
	\begin{equation}\label{eqMsRecursion}
		\#\cM_s^{ij} = \#\cM_{s-1}^{ij} + \#\Aut(\phi_j) \sum_{\substack{\fm\in A_+\\ \deg(\fm)=s}} b_{ij}(\fm), \qquad s\geq 1.
	\end{equation}
	The task is therefore to compute $\sum_{\deg\fm = s} b_{ij}(\fm)$ for $s$ large.
	
	We do this by recognizing the generating function of the $b_{ij}(\fm)$ as a theta series on the Bruhat--Tits tree, decomposing it spectrally, and exploiting the polynomiality of the resulting cuspidal $L$-function. With $\phi_i, \phi_j$ fixed, we drop the double index and write $b(\fm) = b_{ij}(\fm)$. Let $\varpi_\infty$ be a fixed uniformizer at $\infty$. Define Fourier coefficients \cite[p. 735]{WY} (see also \cite[p. 271]{PapikianJNT05})
	\begin{align*}
		c_0(\varpi_\infty^k) &= q^{-k}/w_j, & k&\in \Z, \\
		c(\fm \varpi_\infty^k) &= b(\fm) q^{-k}, & k&\geq 0.
	\end{align*}
	These coefficients determine a unique $\G_0(\fp)$-invariant harmonic cochain $\Theta$ on the Bruhat--Tits tree $\cT$ of $\PGL_2(\Fi)$, where $\G_0(\fp)\subset \GL_2(A)$ is the Hecke congruence subgroup of level $\fp$; see \cite{WY,PapikianJNT05} for the definitions of harmonic cochains and the construction of $\Theta$, and for the proof that $\Theta$ is harmonic and $\G_0(\fp)$-invariant.
	
	The $\C$-vector space $H(\cT,\C)^{\G_0(\fp)}$ decomposes as \cite[Prop. 3.4]{PapikianJNT05}
	\[
	H(\cT,\C)^{\G_0(\fp)} = \C\, E_\fp \oplus H_!(\cT,\C)^{\G_0(\fp)},
	\]
	where $H_!$ denotes the cuspidal subspace and $E_\fp$ is the Eisenstein series with Fourier coefficients
	\begin{align*}
		c_0(E_\fp, \varpi_\infty^k) &= \frac{\abs{\fp}-1}{q^2-1} q^{-k}, & k&\in \Z, \\
		c(E_\fp, \fm \varpi_\infty^k) &= \sigma_\fp(\fm) q^{-k}, & k&\geq 0,
	\end{align*}
	where
	\[
	\sigma_\fp(\fm) = \sum_{\substack{\fm'\in A_+\\ \fm'\mid \fm,\ \fp\nmid \fm'}} \abs{\fm'}.
	\]
	Cusp forms in $H(\cT,\C)^{\G_0(\fp)}$ are characterized by the vanishing of their constant Fourier coefficients $c_0$; hence
	\[
	\Theta_c \colonequals w_j \frac{\abs{\fp}-1}{q^2-1}\Theta - E_\fp
	\]
	is a cusp form, with Fourier coefficients
	\[
	c(\Theta_c, \fm\varpi_\infty^k) = \left(w_j\frac{\abs{\fp}-1}{q^2-1} b(\fm) - \sigma_\fp(\fm)\right) q^{-k} = c(\fm) q^{-k}.
	\]
	
	For a cusp form $f\in H_!(\cT,\C)^{\G_0(\fp)}$, Weil's $L$-function is \cite[p. 273]{PapikianJNT05}
	\[
	L(f,s) = \sum_{\beta\text{ pos. div.}} c(f,\beta) \abs{\beta}^{-s}, \qquad \abs{\beta} = q^{\deg(\beta)},
	\]
	the sum running over all non-negative divisors (including those supported at $\infty$). For $f = \Theta_c$,
	\begin{align*}
		L(\Theta_c, s) &= \sum_{k=0}^\infty \sum_{\fm\in A_+} c(\fm) q^{-k}\abs{\fm}^{-s} q^{-ks} \\
		&= \left(\sum_{\fm\in A_+}\frac{c(\fm)}{\abs{\fm}^s}\right)\sum_{k=0}^\infty q^{-k(1+s)} \\
		&= \frac{1}{1-q^{-(1+s)}}\sum_{\fm\in A_+}\frac{c(\fm)}{\abs{\fm}^s}.
	\end{align*}
	
	When $d\leq 2$ there are no nonzero cusp forms on $\G_0(\fp)$, so $\Theta_c = 0$ and $c(\fm) = 0$ identically. We assume $d\geq 3$ until further notice; the case $d\leq 2$ is handled by the same argument in trivialized form and is recorded at the end of the section.
	
	For $d\geq 3$, deep results of Drinfeld, Deligne, and Grothendieck imply that $L(f,s)$ is a polynomial in $q^{-s}$ of degree $d-3$; see \cite[Ch. 2]{Pap03}. Write
	\[
	L(\Theta_c, s) = a_0 + a_1 q^{-s} + \cdots + a_{d-3} (q^{-s})^{d-3} \in \Z[q^{-s}].
	\]
	Setting $x = q^{-s}$,
	\[
	(a_0 + a_1 x + \cdots + a_{d-3} x^{d-3})(1-q^{-1}x) = \sum_{m=0}^\infty x^m \sum_{\substack{\fm\in A_+\\ \deg\fm = m}} c(\fm),
	\]
	so
	\[
	\sum_{\substack{\fm\in A_+\\ \deg\fm = m}} c(\fm) = 0 \qquad \text{for } m\geq d-1,
	\]
	equivalently
	\begin{equation}\label{eqSumb}
		\sum_{\substack{\fm\in A_+\\ \deg\fm = m}} b(\fm) = \frac{1}{w_j}\cdot \frac{q^2-1}{\abs{\fp}-1} \sum_{\substack{\fm\in A_+\\ \deg\fm = m}}\sigma_\fp(\fm), \qquad m\geq d-1.
	\end{equation}
	
	To turn \eqref{eqSumb} into a usable identity we need an explicit formula for the right-hand side.
	
	\begin{lem}\label{lem_sigmap}
		For every $m\geq 0$,
		\[
		\sum_{\substack{\fm\in A_+\\ \deg\fm = m}}\sigma_\fp(\fm) = \begin{cases}
			\dfrac{q^m(q^{m+1}-1)}{q-1}, & m\leq d-1, \\[1ex]
			\dfrac{q^{2m-d+1}(\abs{\fp}-1)}{q-1}, & m\geq d-1.
		\end{cases}
		\]
	\end{lem}
	
	\begin{proof}
		Write $s(m)$ for the sum, and consider the generating function
		\[
		f(x) = \sum_{\fm\in A_+}\sigma_\fp(\fm) x^{\deg(\fm)} = \sum_{m\geq 0} s(m) x^m.
		\]
		Writing $\fm = \fp^a\fn$ with $a\geq 0$ and $\gcd(\fn,\fp) = 1$, the multiplicativity $\sigma_\fp(\fm) = \sigma_\fp(\fn)$ gives
		\[
		f(x) = \frac{1}{1-x^d}\cdot g(x), \qquad g(x) = \sum_{\substack{\fn\in A_+\\ \gcd(\fn,\fp)=1}} \sigma_\fp(\fn) x^{\deg(\fn)}.
		\]
		By multiplicativity of $\sigma_\fp$ on integers coprime to $\fp$,
		\[
		g(x) = \prod_{\fq\neq \fp}\left(\sum_{e\geq 0}\sigma_\fp(\fq^e) x^{e\deg\fq}\right).
		\]
		For each prime $\fq\neq \fp$ of degree $\delta$, $\sigma_\fp(\fq^e) = (q^{(e+1)\delta}-1)/(q-1)$, so
		\[
		\sum_{e\geq 0}\sigma_\fp(\fq^e) x^{e\delta} = \frac{1}{(1-x^\delta)(1-q^\delta x^\delta)}.
		\]
		Using the standard identities
		\[
		\prod_\fq \frac{1}{1-x^{\deg\fq}} = \frac{1}{1-qx}, \qquad \prod_\fq \frac{1}{1-q^{\deg\fq} x^{\deg\fq}} = \frac{1}{1-q^2 x},
		\]
		we obtain
		\[
		g(x) = \frac{(1-x^d)(1-q^d x^d)}{(1-qx)(1-q^2 x)}, \qquad f(x) = \frac{1-q^d x^d}{(1-qx)(1-q^2 x)}.
		\]
		Partial fractions give
		\[
		\frac{1}{(1-qx)(1-q^2 x)} = \frac{1}{q-1}\left(\frac{q}{1-q^2 x} - \frac{1}{1-qx}\right),
		\]
		so %$[x^m]\bigl((1-qx)(1-q^2 x)\bigr)^{-1} = q^m(q^{m+1}-1)/(q-1)$, and
		\[
		s(m) %= [x^m] f(x) 
		= \frac{q^m(q^{m+1}-1)}{q-1} - q^d \cdot \frac{q^{m-d}(q^{m-d+1}-1)}{q-1}\cdot \mathbf{1}_{m\geq d}.
		\]
		For $m\leq d-1$ the second term vanishes; for $m\geq d$ a direct simplification gives $s(m) = q^{2m-d+1}(q^d-1)/(q-1)$. At the boundary $m=d-1$ both expressions agree.
	\end{proof}
	
	Combining \eqref{eqMsRecursion}, \eqref{eqSumb}, and Lemma~\ref{lem_sigmap}, and using $\#\Aut(\phi_j) = (q-1)w_j$, we obtain
	\begin{equation}\label{eq:cMs-recursion}
		\#\cM_s^{ij} = \#\cM_{s-1}^{ij} + q^{2s-d+1}(q^2-1), \qquad s\geq d-1.
	\end{equation}
	Since $\cM_s^{ij}$ is a finite-dimensional $\F_q$-vector space, $\#\cM_s^{ij} = q^{\dim_{\F_q}\cM_s^{ij}}$. Writing $m_s \colonequals \dim_{\F_q}\cM_s^{ij}$ (omitting the indices $i,j$, fixed throughout), \eqref{eq:cMs-recursion} becomes
	\[
	q^{m_{s-1}}(q^{m_s - m_{s-1}} - 1) = q^{2s-d+1}(q^2-1), \qquad s\geq d-1.
	\]
	Both $q^{m_s-m_{s-1}}-1$ and $q^2-1$ are coprime to $q$, while $q^{m_{s-1}}$ and $q^{2s-d+1}$ are powers of $q$. Equating the two factorizations forces
	\[
	m_{s-1} = 2s - (d-1) \quad \text{and} \quad q^{m_s-m_{s-1}} - 1 = q^2 - 1,
	\]
	hence $m_s - m_{s-1} = 2$. Reindexing ($s-1 \mapsto s$) gives the formula for all $s\geq d-2$.
	
	\begin{thm}\label{thmMain}
		For every $1\leq i, j\leq n$ and every $s\geq d-2$,
		\[
		\dim_{\F_q}\cM_s^{ij} = 2(s+1) - (d-1).
		\]
		In particular, $\dim_{\F_q}\cM_{d-2}^{ij} = d-1$.
	\end{thm}
	
	\begin{rem}\label{rmk:small-d}
		For $d\leq 2$, the cuspidal subspace $H_!(\cT,\C)^{\G_0(\fp)}$ is zero, so $\Theta = \frac{q^2-1}{w_j(\abs{\fp}-1)} E_\fp$ exactly and $c(\fm) = 0$ for all $\fm$. The vanishing $\sum_{\deg\fm = m} c(\fm) = 0$ in \eqref{eqSumb} now holds for all $m\geq 0$ trivially, and the same recursion \eqref{eq:cMs-recursion} runs from $s = 1$, giving $\dim_{\F_q}\cM_s^{ij} = 2(s+1) - (d-1)$ for all $s \geq 0$.
		
		In these small-degree cases there is a single supersingular isomorphism class over $\oF_\fp$, so $n = \mathbf{t} = 1$ and the indices $i,j$ disappear:
		\begin{itemize}
			\item if $d = 1$, the unique class is represented by $\phi_T = t + \tau^2$, and $\dim_{\F_q}\cM_s = 2(s+1)$ for all $s \geq 0$;
			\item if $d = 2$, the unique class admits the $\F_\fp$-rational representative
			\[
			\phi_T = t + \tau + \frac{1}{t^q - t}\tau^2,
			\]
			obtained as the Legendre form $\phi_T = t + \tau + x_0\tau^2$ at the unique root $x_0 \in \F_\fp$ of $H_\fp(x)$, and $\dim_{\F_q}\cM_s = 2s + 1$ for all $s \geq 0$.
		\end{itemize}
	\end{rem}
	
	%==================================================================

%==================================================================
\section{Examples}\label{sExample}

We collect explicit computations of $\cM_s(\phi,\psi)$ in two settings. Throughout this section we abbreviate
\[
m_s(\phi,\psi) \colonequals \dim_{\F_q}\cM_s(\phi,\psi), \qquad m_s(\phi) \colonequals m_s(\phi,\phi).
\]
Subsection~\ref{ssDiagonal} treats the special case $\phi_T = t + \tau^r$, where the stabilization formula and an explicit successive-minima basis admit closed forms via the Sylvester--Frobenius theorem. Subsection~\ref{ssCompBM} describes an algorithm for computing $\cM_s(\phi,\psi)$ in rank $2$ and applies it to enumerate, for two specific primes, the pre-stable dimension tuples $(m_0(\phi),\ldots,m_{d-2}(\phi))$ across all supersingular isomorphism classes. The observed tuples are in bijection with the SMB multisets allowed by the constraints of Section~\ref{sec:smb-stabilization}, providing an experimental confirmation that those constraints are not only necessary but realized.

%==================================================================
\subsection{The case \texorpdfstring{$\phi_T = t + \tau^r$}{phi\_T = t + tau\textasciicircum r}}\label{ssDiagonal}

Throughout this subsection $r\geq 2$ is an integer, $g = \gcd(r,d)$, and $\phi$ denotes the Drinfeld module of rank $r$ over $\F_\fp$ defined by
\[
\phi_T = t + \tau^r.
\]
We derive a stabilization theorem for $m_s(\phi)$ by elementary means and exhibit an explicit successive-minima basis for $\End(\phi)$, recovering by direct calculation the formula of Section~\ref{sec:smb-stabilization}.

\begin{lem}\label{lemDSpecial}
	We have $H(\phi) = r/g$. In particular, $\phi$ is supersingular if and only if $g = 1$.
\end{lem}

\begin{proof}
	Let $\Phi$ be the Drinfeld module over $F$ defined by $\Phi_T = T + \tau^r$, so that $\phi$ is the reduction of $\Phi$ modulo $\fp$. Set $A' = \F_{q^r}[T]$ and $F' = \F_{q^r}(T)$. The prime $\fp$ decomposes in $A'$ as
	\[
	\fp = \fP_1 \cdots \fP_g,
	\]
	where each $\fP_i$ has degree $d/g$. Viewing $\Phi$ as a Drinfeld module over $F'$, its reduction modulo $\fP_1$ is isomorphic to $\phi$ over the degree-$r/g$ extension of $\F_\fp$, via the natural embedding $A/\fp \hookrightarrow A'/\fP_1$.
	
	On the other hand, $\Phi$ over $F'$ is a Drinfeld $A'$-module of rank $1$, and the reduction of a rank-$1$ Drinfeld module has height $1$. Therefore $\Phi_{\fP_1} \equiv \tau^{rd/g} \pmod{\fP_1}$. For $2\leq i\leq g$, the constant term of $\Phi_{\fP_i}$ is $\fP_i$, nonzero modulo $\fP_1$. Letting $\alpha\neq 0$ be the product of the images of $\fP_2,\ldots,\fP_g$ in $A'/\fP_1$,
	\[
	\Ht(\phi_\fp) = \Ht(\Phi_\fp \bmod \fP_1) = \Ht\bigl(\alpha \tau^{rd/g} + \text{higher-degree terms}\bigr) = rd/g.
	\]
	Hence $H(\phi) = r/g$.
\end{proof}

\begin{rem}
	Lemma~\ref{lemDSpecial} strengthens \cite[Example 4.4.5]{PapikianGTM} and simplifies the argument there.
\end{rem}

\begin{lem}\label{lemExplicitEnd}
	Assume $g = 1$, so that $\phi$ is supersingular. Then $\End(\phi) = \F_{q^r}[\phi_T, \tau^d]$.
\end{lem}

\begin{proof}
	Again set $A' = \F_{q^r}[T]$ and $F' = \F_{q^r}(T)$. The inclusion $\F_{q^r}[\phi_T,\tau^d]\subseteq \End(\phi)$ is clear. The ring $\F_{q^r}[\phi_T,\tau^d]$ is isomorphic to $\cO\colonequals A'[\pi]/(\pi^r - \fp)$, with the commutation relation
	\[
	\pi(f_0 + f_1 T + \cdots + f_s T^s) = (f_0^{q^d} + f_1^{q^d}T + \cdots + f_s^{q^d}T^s)\pi.
	\]
	Since $g = 1$, the automorphism of $F'$ induced by $\pi$ generates $\Gal(F'/F)$, so $D\colonequals \cO\otimes_A F$ is a cyclic $F$-algebra of dimension $r^2$ and $\cO$ is an $A$-order in $D$. A direct computation shows $\disc(\cO/A) = \fp^{r(r-1)}$, hence $\cO$ is maximal in $D$; see \cite[Sec.~6]{BG}. Therefore $\F_{q^r}[\phi_T,\tau^d] = \End(\phi)$.
\end{proof}

\begin{prop}\label{proptaur}
	Assume $g = 1$. Then for all $s\geq (r-1)(d-1) - 1$,
	\[
	m_s(\phi) = r(s+1) - \frac{r(r-1)(d-1)}{2}.
	\]
\end{prop}

\begin{proof}
	We have $\cM_0(\phi) = \F_{q^r}$, so $m_0(\phi) = r$. For any $s\geq 0$, any difference of two elements of $\tau$-degree $s+1$ in $\End(\phi) = \F_{q^r}[\phi_T,\tau^d]$ lies in $\cM_s$ (after scaling), so $m_{s+1}(\phi) \leq m_s(\phi) + r$, with equality if and only if $\End(\phi)$ contains an element of $\tau$-degree exactly $s+1$, i.e., if and only if $s+1$ is representable as $nr + md$ with $n,m\geq 0$.
	
	Since $\gcd(r,d) = 1$, the Sylvester--Frobenius theorem \cite{Sylvester1884} states that every integer $\geq (r-1)(d-1)$ is representable, and exactly $(r-1)(d-1)/2$ integers in \[\{0,1,\ldots,(r-1)(d-1)-1\}\] are representable (counting $0$). Subtracting the contribution of $0$, the dimension $m_s(\phi)$ increases by $r$ at exactly $(r-1)(d-1)/2 - 1$ values of $s$ in the range $1 \leq s \leq (r-1)(d-1) - 1$, and at every subsequent $s$. Combining with $m_0(\phi) = r$,
	\begin{align*}
			&m_{(r-1)(d-1)-1}(\phi) = \frac{r(r-1)(d-1)}{2},\\
		& m_s(\phi) = m_{s-1}(\phi) + r \text{ for } s\geq (r-1)(d-1),
	\end{align*}
	which is the stated formula.
\end{proof}

\begin{prop}\label{prop:phi-vanishing}
	Assume $g = 1$ and set $\La \colonequals \End(\phi)$. The multiset of $\tau$-degrees of an SMB of $\La$ is
	\[
	\{bd : 0\leq b\leq r-1\}, \qquad \text{each value appearing with multiplicity } r.
	\]
	In particular, $e_{\max}(\La) = (r-1)d$, and Conjecture~\ref{conj:emax-bound} holds for $\La$ with equality.
\end{prop}

\begin{proof}
	By Lemma~\ref{lemExplicitEnd}, $\La = \F_{q^r}[\phi_T,\tau^d]$. As an $A$-module (with $T$ acting via $\phi_T$), $\La$ has the basis
	\[
	\{\zeta_a\tau^{bd} : 0\leq a\leq r-1,\ 0\leq b\leq r-1\},
	\]
	where $\zeta_0 = 1,\zeta_1,\ldots,\zeta_{r-1}$ is an $\F_q$-basis of $\F_{q^r}\subset \La$. Each basis element has $\deg_\tau = bd$.
	
	We show this basis is norm-orthogonal, equivalently an SMB by \cite[Lem.~4.2]{Taguchi}. The reduced norm of $\tau^d$ is $\Nr(\tau^d) = \fp$: indeed, $\Ht(\tau^d) = d$ and $\tau^d$ has trivial kernel as an $A$-module scheme, so \cite[Lem.~3.10]{Gekeler91} gives $\Nr(\tau^d) = \fp^{d/d} = \fp$. Hence $w(\tau^d) = -d/r$.
	
	Work in the cyclic-algebra description $D_\infty = F_{\infty,r}\langle\pi\rangle$, where $F_{\infty,r}$ is the unramified extension of $\Fi$ of degree $r$, $\pi^r = T^{-1}$, and $\pi a = \sigma(a)\pi$ for $a\in F_{\infty,r}$, with $\sigma$ the generator of $\Gal(F_{\infty,r}/\Fi)$ lifting the $\F_q$-Frobenius on $\F_{q^r}$. The valuation extends as $w(\sum_{b'=0}^{r-1} \alpha_{b'}\pi^{b'}) = \min_{0\leq b'\leq r-1}(\ord_\infty(\alpha_{b'}) + b'/r)$, with the ultrametric inequality an equality because the values have pairwise distinct fractional parts mod $1$.
	
	For $c_{a,b}\in \Fi$, group $\sum_{a,b}c_{a,b}\zeta_a\tau^{bd} = \sum_b C_b\tau^{bd}$ with $C_b = \sum_a c_{a,b}\zeta_a \in F_{\infty,r}$. Since $F_{\infty,r}/\Fi$ is unramified and $\{\zeta_a\}$ reduces to an $\F_q$-basis of $\F_{q^r}$, the elements $\zeta_a$ are units whose reductions are $\F_q$-linearly independent, so $\ord_\infty(C_b) = \min_a\ord_\infty(c_{a,b})$. Hence
	\[
	w(C_b\tau^{bd}) = \min_a\ord_\infty(c_{a,b}) - bd/r.
	\]
	The fractional parts $-bd/r \bmod 1$ for $b\in\{0,\ldots,r-1\}$ are distinct (using $\gcd(r,d) = 1$), so $w(\sum_b C_b\tau^{bd}) = \min_b w(C_b\tau^{bd}) = \min_{a,b}(\ord_\infty(c_{a,b}) - bd/r)$. Translating to $\norm{\cdot}_D = q^{-w(\cdot)}$,
	\[
	\Bigl\|\sum_{a,b}c_{a,b}\zeta_a\tau^{bd}\Bigr\|_D = \max_{a,b}\abs{c_{a,b}}\cdot q^{bd/r} = \max_{a,b}\abs{c_{a,b}}\cdot \norm{\zeta_a\tau^{bd}}_D,
	\]
	which is norm-orthogonality. Hence $\{\zeta_a\tau^{bd}\}$ is an SMB of $\La$ with minima $\{bd:0\leq b\leq r-1\}$, each with multiplicity $r$; in particular $e_{\max}(\La) = (r-1)d$.
\end{proof}

%==================================================================
\subsection{Explicit computations in rank \texorpdfstring{$2$}{2}}\label{ssCompBM}

Let $\phi$ and $\psi$ be supersingular Drinfeld modules of rank $2$ over $\oF_\fp$. We describe an algorithm for computing
\[
\cM_s(\phi,\psi) = \{u \in \Hom(\phi,\psi) : \deg_\tau(u) \leq s\}
\]
as a subspace of $\twist{\oF_\fp}$, then apply it to enumerate the pre-stable dimension tuples $(m_0(\phi),\ldots,m_{d-2}(\phi))$ across all supersingular isomorphism classes for two specific primes. The algorithm extends to higher ranks with cosmetic changes; we restrict to rank $2$ for concreteness.

Choose convenient models. If $j(\phi) \neq 0$, take $\phi_T = t + \tau + x\tau^2$ with $x$ a root of $H_\fp(x)$ (cf. \eqref{eqLegendre}); if $j(\phi) = 0$, take $\phi_T = t + \tau^2$. By Proposition~\ref{lemLegendre}, $\phi$ and all its endomorphisms are defined over $\F_{\fp^2}$, and all isogenies $\phi \to \psi$ are defined over $\F_{\fp^2}$ in many cases, and over $\F_{\fp^{2(q-1)}}$ in the worst case. Let $k = \F_{\fp^2}$ when computing endomorphisms, and $k = \F_{\fp^{2(q-1)}}$ when computing isogenies between non-isomorphic Drinfeld modules.

Write a general element $u \in \cM_s(\phi,\psi)$ as $u = \sum_{i=0}^s u_i\tau^i$ with $u_i \in k$, the $u_i$ being unknowns. The products $u\phi_T$ and $\psi_T u$ expand as skew polynomials in $\tau$ supported in degrees $\leq s+2$; equating coefficients of $\tau^j$ in $u\phi_T = \psi_T u$ for $0 \leq j \leq s+2$ yields a homogeneous linear system in the $u_i$ and their Frobenius twists. Fix an $\F_q$-basis of $k$ and identify each $u_i$ with an $m$-tuple over $\F_q$ ($m = [k:\F_q]$). The commutation relations become a homogeneous $\F_q$-linear system with $(s+1)m$ unknowns and $(s+3)m$ equations, whose solution space is $\cM_s(\phi,\psi)$.

For increasing $s$ we compute the spaces $\cM_0 \subset \cM_1 \subset \cM_2 \subset \cdots$ via nested bases: a basis of $\cM_{s-1}(\phi,\psi)$ is embedded into the ambient space for $\cM_s(\phi,\psi)$ by appending $m$ zero coordinates, then extended to a full basis of $\cM_s(\phi,\psi)$. The resulting nested $\F_q$-basis is essential for enumerating isogenies of exact $\tau$-degree $s$, and the values $b_{ij}(\fm)$ of the Brandt matrices can be read off by running this procedure for each pair $(\phi_i,\phi_j)$ and partitioning $\cM_s(\phi_i,\phi_j)$ by reduced-norm class.

We have implemented this algorithm in Magma~\cite{Magma}; the computations in Examples~\ref{ex:d5} and \ref{ex:d4} were carried out using this implementation. We have applied this algorithm to enumerate the pre-stable tuples $(m_0(\phi),\ldots,m_{d-2}(\phi))$ for several small $(q,\fp)$. We present in detail two cases that illustrate the typical phenomena. In both, the observed tuples turn out to be in bijection with the SMB multisets permitted by the constraints of Section~\ref{sec:smb-stabilization}; this gives a sharp experimental confirmation that those constraints are tight.

\begin{example}\label{ex:d5}
	Let $q = 3$ and $\fp = T^5 + 2T + 1$, so $d = 5$ and $\#\F_\fp = 243$. The supersingular polynomial $H_\fp(x)$ has degree $(q^d - q)/(q^2 - 1) = 30$ and factors over $\F_\fp$ as a product of $16$ linear factors and $7$ irreducible quadratics. Together with the class $\phi_T = t + \tau^2$, supersingular by Lemma~\ref{lemDSpecial} since $d$ is odd, this gives $n = 31$ supersingular isomorphism classes over $\oF_\fp$.
	
	Computing $m_s(\phi)$ for $s = 0,1,2$ across all $31$ classes (working over $k = \F_{\fp^2}$) yields exactly four distinct triples $(m_0,m_1,m_2)$, distributed as follows:
	\begin{center}
		\begin{tabular}{lccc}
			\hline
			$\phi$ & \# classes & $(m_0,m_1,m_2)$ & SMB multiset $\{e_k\}$ \\
			\hline
			$\phi_T = t + \tau^2$ & $1$ & $(2,2,4)$ & $(0,0,5,5)$ \\
			$x_0 \in \F_\fp$, exceptional & $3$ & $(1,2,3)$ & $(0,1,4,5)$ \\
			$x_0 \in \F_\fp$, generic & $13$ & $(1,1,3)$ & $(0,2,3,5)$ \\
			$x_0 \in \F_{\fp^2}\setminus\F_\fp$ & $14$ & $(1,1,2)$ & $(0,3,3,4)$ \\
			\hline
		\end{tabular}
	\end{center}
	Here we call $x_0 \in \F_\fp$ \emph{exceptional} if $\End(\phi)$ contains an isogeny of $\tau$-degree exactly $1$, equivalently $e_2(\End(\phi)) = 1$; three of the $16$ $\F_\fp$-rational roots have this property. The fourteen roots in $\F_{\fp^2}\setminus\F_\fp$ form seven Galois-conjugate pairs, each pair giving a pair of Drinfeld modules with identical SMB invariants.
	
	We verify the SMB multisets independently from the constraints of Section~\ref{sec:smb-stabilization}. For $r = 2$ and $d = 5$ these constraints are
	\begin{enumerate}
		\item[(1)] $\sum_k e_k = r^2(r-1)d/2 = 10$ (Proposition~\ref{prop:sum-ek});
		\item[(2)] residue equidistribution mod $2$, each class represented twice (Proposition~\ref{prop:uniform-residues});
		\item[(3)] $e_{\max} \leq (r-1)d = 5$ (Corollary~\ref{cor:emax-r2});
		\item[(4)] $e_1 = 0$ (Lemma~\ref{lem:e1-zero}).
	\end{enumerate}
	A direct enumeration shows that (1)--(4) admit exactly four solutions:
	\[
	(0,0,5,5), \quad (0,1,4,5), \quad (0,2,3,5), \quad (0,3,3,4),
	\]
	matching the four observed SMB multisets. Each multiset yields a distinct triple $(m_0,m_1,m_2)$ via Theorem~\ref{thm:dimMs-elementary}, recovering the four observed triples in the order shown.
	
	The matching extends to identifying the underlying Drinfeld module. The multiset $(0,0,5,5)$ is the only one with two zero minima; by Lemma~\ref{lem:e1-zero}, this characterizes the class with $\ell(\phi) = 2$, i.e., $\phi_T = t + g\tau^2$, equivalently $\phi \cong t + \tau^2$. By Proposition~\ref{prop:phi-vanishing} this SMB shape is forced. The other three multisets, distinguished by the value of $e_2 \in \{1,2,3\}$, partition the $30$ classes with $j(\phi)\neq 0$.
\end{example}

\begin{example}\label{ex:d4}
	Let $q = 3$ and $\fp = T^4 + T + 2$, so $d = 4$ and $\#\F_\fp = 81$. Now $H_\fp(x)$ has degree $(q^d - 1)/(q^2 - 1) = 10$ and factors over $\F_\fp$ as a product of $4$ linear factors and $3$ irreducible quadratics. Since $d$ is even, $\phi_T = t + \tau^2$ is \emph{not} supersingular by Lemma~\ref{lemDSpecial}, so the supersingular classes are precisely the $n = 10$ classes corresponding to roots of $H_\fp$.
	
	The pre-stable range is $s\in\{0,1\}$. Computing $(m_0(\phi), m_1(\phi))$ across all $10$ classes yields two distinct pairs:
	\begin{center}
		\begin{tabular}{lccc}
			\hline
			$\phi$ & \# classes & $(m_0,m_1)$ & SMB multiset $\{e_k\}$ \\
			\hline
			$x_0 \in \F_\fp$ & $4$ & $(1,2)$ & $(0,1,3,4)$ \\
			$x_0 \in \F_{\fp^2}\setminus\F_\fp$ & $6$ & $(1,1)$ & $(0,2,3,3)$ \\
			\hline
		\end{tabular}
	\end{center}
	The six roots in $\F_{\fp^2}\setminus\F_\fp$ form three Galois-conjugate pairs. The constraints (1)--(4) of Example~\ref{ex:d5}, now with $\sum_k e_k = 8$ and $e_{\max}\leq 4$, admit exactly the two solutions $(0,1,3,4)$ and $(0,2,3,3)$; each gives a distinct pair $(m_0,m_1)$ via Theorem~\ref{thm:dimMs-elementary}, matching the table.
\end{example}

\begin{comment}\begin{rem}
	In both examples, the SMB multisets allowed by the constraints of Section~\ref{sec:smb-stabilization} are not only realized but realized by an explicit class of Drinfeld modules characterized by arithmetic conditions on the parameter $x_0$ (rationality over $\F_\fp$, the value of $e_2$). The structural reason for the trichotomy $e_2 \in \{1,2,3\}$ observed in Example~\ref{ex:d5}, or for the dichotomy $e_2 \in \{1,2\}$ in Example~\ref{ex:d4}, is not addressed here.
	\end{rem}
	\end{comment}

	%---------------------------------------------------------
	%\renewcommand{\bibliofont}{\normalsize}
	\bibliographystyle{amsalpha}
	\bibliography{bibliography.bib}
	
\end{document}